\def\version{\today}	        	%

\documentclass[reqno,11pt]{amsart} 
\usepackage[T1]{fontenc}
\usepackage[utf8]{inputenc}
\usepackage{amsmath,amsthm} 
\usepackage[normalem]{ulem}

\usepackage[authoryear]{natbib}

\usepackage{fontawesome5}

\usepackage{mathrsfs}
\usepackage{amssymb}
\usepackage{srcltx} 
\usepackage{dsfont}
\usepackage{hyperref}
\usepackage{color}
\usepackage{enumerate}
\usepackage{tikz, pgfplots}
\usepackage{caption}
\usepackage{subcaption}
\usepackage{comment}
\usepackage{lscape}
\usepackage{graphicx}
\usepackage{tabularx}
\usepackage{pdfpages}    % Felix: used in Figure fig:THEproof-illustration to draw with tikz above an imported picture

\numberwithin{equation}{section}
 
\def\b{\beta}

\def\d{{\rm d}}

\newfam\Bbbfam 
\font\tenBbb=msbm10 
\font\sevenBbb=msbm7 
\font\fiveBbb=msbm5 
\textfont\Bbbfam=\tenBbb 
\scriptfont\Bbbfam=\sevenBbb 
\scriptscriptfont\Bbbfam=\fiveBbb

\def\2{\mathbf 2}

\newcommand{\R}     {\mathbb{R}} 
\newcommand{\Z}     {\mathbb{Z}} 
\newcommand{\N}     {\mathbb{N}} 
\renewcommand{\P}   {\mathbb{P}} 
 
\newcommand{\E}     {\mathbb{E}}

\def\1{{\mathchoice {1\mskip-4mu\mathrm l}      % Blackboard bold 1 
{1\mskip-4mu\mathrm l} 
{1\mskip-4.5mu\mathrm l} {1\mskip-5mu\mathrm l}}} 
 
\def\comment#1{} 
\newtheoremstyle{thm}{2ex}{2ex}{\itshape\rmfamily}{} 
{\bfseries\rmfamily}{}{1.7ex}{} 
 
\newtheoremstyle{rem}{1.3ex}{1.3ex}{\rmfamily}{} 
{\itshape\rmfamily}{}{1.5ex}{}

\newcommand\numberthis{\addtocounter{equation}{1}\tag{\theequation}}
 
\renewcommand{\theequation}{\thesection.\arabic{equation}} 
 
\newtheorem{theorem}{Theorem}[section] 
\newtheorem{lemma}[theorem]{Lemma} 
\newtheorem{prop}[theorem] {Proposition} 
\newtheorem{cor}[theorem]  {Corollary}

\newtheorem{conj}[theorem] {Conjecture}

\theoremstyle{definition}
\newtheorem{definition}[theorem] {Definition}

\newtheorem{remark}[theorem]{Remark}

\renewcommand{\d}{{\rm d}} 
 
\newcommand{\eps}{\varepsilon}

\definecolor{Red}{rgb}{1,0,0}

 \pgfplotsset{compat=1.18}

\title[Logarithmic scaling of selective sweep curves]{Logarithmic scaling of selective sweep curves:\\ from tents to houses}
\author[F.~Boenkost, F.~Hermann, A.~Tóbiás and A.~Wakolbinger]{}

\begin{document}

\begin{abstract}
One of the classical results of mathematical population genetics states that the frequency of a beneficial mutant's offspring, on its way to fixation in a large population, looks like a logistic curve. A {\em logarithmic scaling} (both in height and time) of these {\em selective sweep curves} leads (in the case of strong selection) to a tent-like shape in the large population limit: first the logarithmic frequency of the mutant increases linearly from 0 to 1, then that of the former resident decreases from 1 to 0. For {\em moderate selection} the logarithmic frequencies develop (in the large population limit) a jump at the beginning/the end of the sweep, which takes the shape of the tent into that of a house.
Our main result (proved for the Moran model) assesses the regularity of this convergence in the large population limit: it is uniform in the house's roof
(phases of linear growth and decline)
and ``Skorokhod $M_1$'' in the house's walls
(closely around the jumps).
Apart from interest in its own right, we anticipate that this result and the proof techniques will be instrumental for extending the description of clonal interference by Poissonian interacting trajectories (as it was done in~\cite{HGSTW24} for strong selection) also to moderate selection.
\end{abstract}

\maketitle

\vspace{-0.5cm}

\begin{center}
  {\sc Florin Boenkost\footnote{\mbox{Universität Wien, Fakultät für Mathematik, Oskar-Morgenstern-Platz 1, Wien, Austria, {\tt florin.boenkost@univie.ac.at}}}
  Felix Hermann\footnote{Goethe-Universität Frankfurt am Main, FB 12, Institut für Mathematik, 60629 Frankfurt, Germany, {\tt hermann@math.uni-frankfurt.de}},
   András Tóbiás\footnote{Department of Computer Science and Information Theory, Faculty of Electrical Engineering and Informatics, Budapest University of Technology and Economics, Műegyetem rkp. 3., H-1111 Budapest, and HUN-REN Alfréd Rényi Institute of Mathematics, Reáltanoda utca 13--15, H-1053 Budapest, Hungary, {\tt tobias@cs.bme.hu}}
   and
   Anton Wakolbinger\footnote{Goethe-Universität Frankfurt am Main, FB 12, Institut für Mathematik, 60629 Frankfurt, Germany, {\tt wakolbin@math.uni-frankfurt.de}}}
\end{center}

\centerline{\small(\version)}


\maketitle

\section{Introduction}\label{secIntro}
In a population whose total size $N$ is kept constant over the generations, consider the evolution of the size of that subpopulation which has inherited a beneficial mutation, with selective advantage quantified by $s>0$, that appeared in a single individual at a certain selective locus at time 0. If the relative frequency of this {\em mutant subpopulation} rises quickly (on the ecological timescale) from $1/N$ to $1$ then one speaks of a {\em selective sweep}. This phenomenon is closely related to that of {\em genetic hitchhiking} (a term coined by ~\cite{smith1974hitch}): in a selective sweep, the mutant individual's types at neutral loci that are close to the selective locus (and thus are less affected by recombination) are taken to fixation along with the mutant type at the selective locus.   In this way part of the genetic diversity near the selective locus is ``swept away''. For more biological background on selective sweeps we  refer to the survey paper~\cite{stephan2019selective}. Mathematical aspects of the spread of a favorable mutant in a finite population in connection with branching process approximation are discussed in~\cite{burger2026bounds}.

The relative frequency of the mutant's offspring, conditioned to the event of fixation and observed between the times when it reaches $\varepsilon$ resp. $1-\varepsilon$ for the first time looks roughly like a logistic curve.
This is how~\cite{KHL89} approximated the middle piece of the sweep, while they approximated the initial piece by a supercritical branching process (of the mutant population) and the final piece by a subcritical  branching process (of the  resident population). \cite{barton1998effect} further divided the initial piece into an {\em onset phase} in which the (then still very small) mutant subpopulation has to survive the fluctuations, and an {\em exponential growth phase}. Here, a crucial insight is that this ``conditioning to survival'' significantly shortens the time until the mutant subpopulation reaches the size above which its fixation has high probability. In this context we quote, with a cosmetic modification consisting in the removal of two factors 2, a sentence from~\cite{barton1998effect}:   
{\em In particular, the expected frequency
of an allele given that it will fix is accelerated above the
unconditional expectation of $e^{st}/N$ by a factor~$1/s$.} The modification  results from replacing the population size $2N$ (used by Barton) by $N$  and from using the generation time of the Moran model, which in our scaling runs twice as fast as that of the standard Wright--Fisher model.
\normalcolor

A naive guess for the duration of the initial piece of the sweep is $\frac 1s\log N$, arising as the time a Malthusian growth with rate $s$ starting from size $1$, i.e.\ the function $t\mapsto e^{st}$, takes to reach $N$.
However, {Barton's} observation  suggests a   duration $\frac 1s\log (Ns)$ instead of $\frac 1s\log (N)$. This has been discussed in \cite{baake2019modelling} in the context of Lenski's long-term experiment, see in particular formulas (19) and (20) therein. The core of the mathematical heuristics is as follows: the expectation of the mutant population at time $t$, conditioned under its surviving the fluctuations, is $\frac1se^{st}$. Equating this to $\frac1s\log N$ (which is a threshold above which fixation of the mutant population happens with high probability even without conditioning) and solving for $t$ gives $t=\frac1s\log\log N$.
After reaching size $\frac1s\log N$, as survival is ensured, Malthusian growth starts and reaches $N$ after a time of order $\frac1s\log(sN)$. Since $\frac1s\log(sN) \gg \frac1s\log\log N$, these heuristics suggest an initial jump to the height $\frac1s\log N$ on the timescale of the duration of the sweep, indeed observable in Figure~\ref{fig:single_invader} (bottom left).  (Note that this change of duration has no effect, i.e.\ $\frac1s\log N\sim\frac1s\log(sN)$, if and only if $\log_N(s)\to0$.) 

Like that of the initial piece, also the duration of the final piece of the sweep is $~ \frac 1s \log(sN)$. This can be heuristically explained in a similar vein, by dividing the final piece into a phase dominated by Malthusian decay, and a phase in which genetic drift dominates, now taking the resident to quick extinction. More specifically, Malthusian decay of the original resident's size effectively starts as soon as the mutant has reached a size of order close to $N$. This decay brings the original resident down to, say, size $\frac 1s \sqrt{\log N}$ in a time $t \sim \frac 1s \log(sN)$. Once this size is reached, the time to extinction is negligible compared to $t \sim \frac 1s \log N$. A quick heuristic argument to see this is a comparison with neutral evolution: it is well known (see \cite{ewens2004mathematical}, \cite{durrett2008probability}) that in a neutral (Wright--Fisher or Moran) model a subpopulation starting in size $k \ll N$ gets extinct with high probability in a time of order $k$. The choice $k:= \frac 1s \sqrt{\log N}$ makes this time negligible compared to $\frac 1s \log N$, and suggests a final jump from the height  $\frac 1s \log N$ to $0$  on the timescale of the sweep duration. 
These considerations suggest to extend Nick Barton's description of the ``four phases of a selective sweep'' in~\citep{barton1998effect} to five phases analyzed in Propositions~\ref{Zph1}--~\ref{Zph5} below. The additon of this fifth also vanishing phase makes the picture of the sweep symmetric in time, cf.\ Figure~\ref{fig:single_invader} (bottom left).

The present work is a mathematical study of the stochastic process performed by the sizes $(X^N_0, X^N_1)$ of the resident and the mutant population in a selective sweep which makes the above-described heuristics rigorous. As our mathematical framework we will use the (continuous-time) two-type Moran model with the two parameters $N$ (the total population size) and $s$ (the coefficient of directional selection).  Specifically, the number $X_1^N$ of individuals carrying the beneficial type performs a birth-and-death process with jump rate   $(1+s) k(N-k)$ from $k$ to $k+1$, and jump rate $k(N-k)$ from $k$ to $k-1$. We will think of $X_1^N$ as the size of the offspring of a single mutant that enters at time $t=0$, so that $X_1^N(0) = 1$, and denote by $X_0^N(t) := N-X_1^N(t)$ the size of the  resident (or wild type) population. 

We will make intensive use of state-dependent time-changes which in the early phase of the sweep take $X_1^N$ into a supercritical and in the final phase of the sweep take $X_0^N$ into a subcritical branching process. It is common knowledge that classical models from population genetics (like the Wright--Fisher, the Moran and the Cannings model) in suitable scalings admit branching process approximations as long (or as soon) as one of the types is rare (see e.g.\ \cite{ewens2004mathematical}, \cite{durrett2008probability}, \cite{B21i}). The Moran model provides a particularly convenient mathematical framework for this approach. We conjecture, however, that appropriate versions of our main result remain valid for any family of Cannings models with population size $N$, selection strength $s$ and (neutral) offspring variance $\sigma^2$ per individual per generation, after adjusting the timescale by a factor $\tfrac {\sigma^2}2$. This comprises the standard Wright--Fisher model for which $\sigma^2$ is (close to)~1. 

In the next paragraphs we will briefly comment on the regimes of moderate and strong selection. As it turns out, in these regimes an appropriate logarithmic scaling of the selective sweep curves $(X_0^N(t),X_1^N(t))$ leads in the regime of strong selection to the shape of a tent, and in that of moderate selection to the shape of a house, with both shapes displaying five phases of the sweep.

\normalcolor
\subsubsection*{Strong selection} In this regime it is assumed that, as $N\to \infty$, the parameter $s$ does not scale with~$N$, so that the effect of selection remains visible on the {\em generation timescale} also for large populations $N$. It is known  (see~e.g.~\cite[Proposition~3.1]{C17Markovproc} and~\cite[Theorem~3.1]{kurtz1971limit}) that under this assumption, with 
\begin{equation}\label{deftaueps}
  \mathcal T^N_\eps:= \inf\{t>0: X_1^N (t) \ge \eps N\}  
\end{equation}
conditional on the event $\{\mathcal T_1^N< \infty\}$, for each $\eps \in (0,1)$ the sequence of processes $(\tfrac 1NX_1^N(\mathcal T^N_\eps+t))_{t\ge 0}$    converges uniformly in probability (as $N\to \infty$) to 
\begin{equation}\label{logcurve}
x_\eps(t):= \frac 1{1+e^{-st}(1/\eps -1)},\quad t\ge 0,
\end{equation}
the solution of the logistic differential equation
\begin{equation}\label{logistic}
    \frac{\d x(t)}{\d t} = s\,x(t)(1-x(t)), \quad x(0) = \eps.
\end{equation}

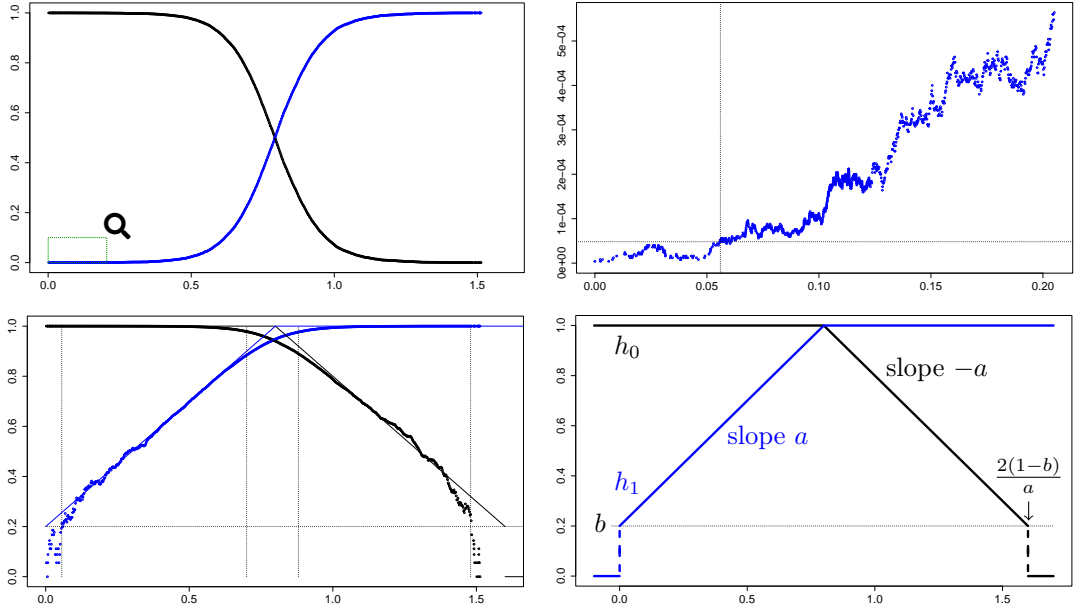
\begin{figure}
 \begin{tikzpicture}
    \node at (0,0) {\includegraphics[width=7cm,trim={0.4cm, 1.5cm, 0.8cm, 2cm},clip]{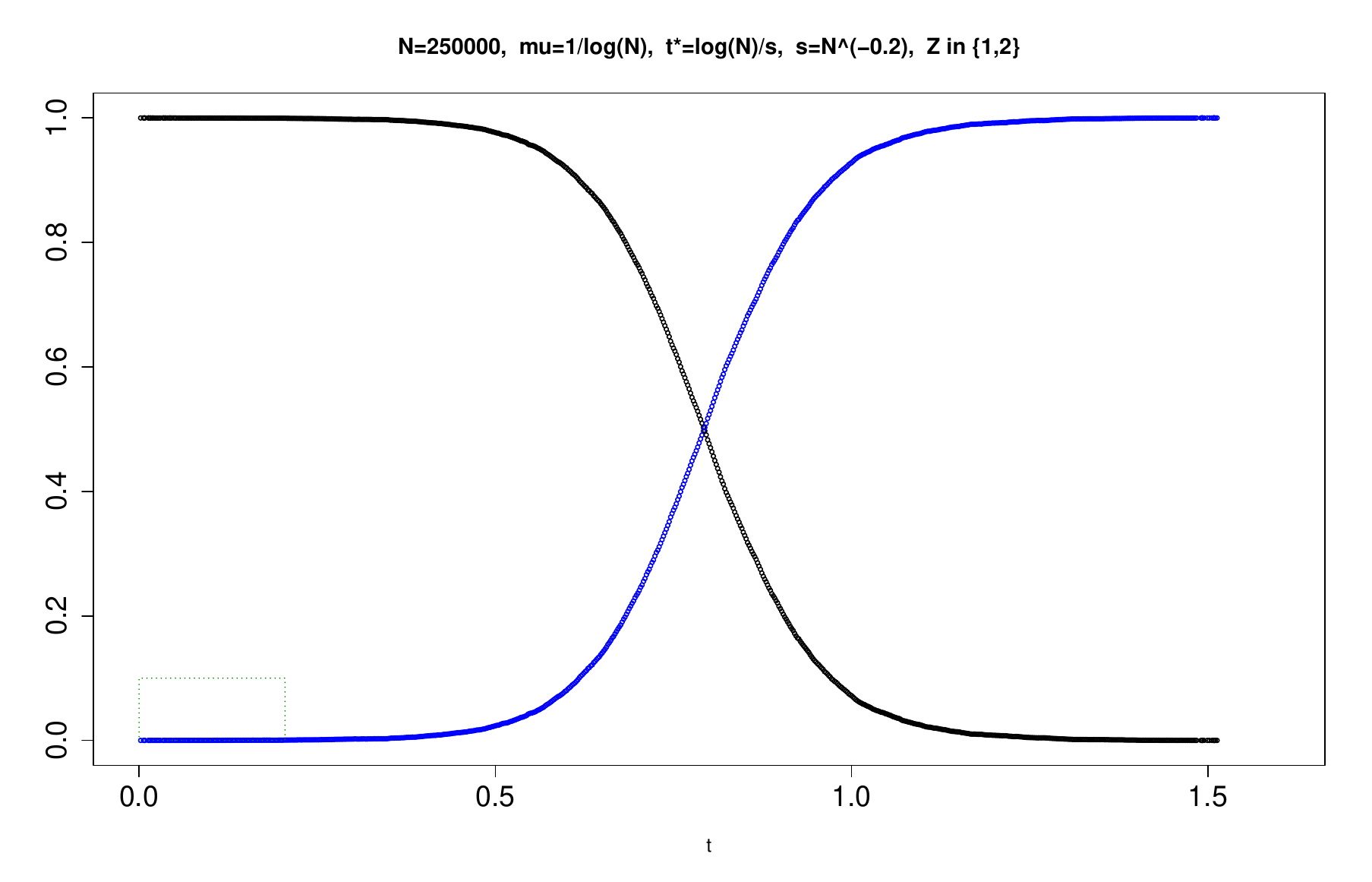}};
    \node[right,black] at (-2.25,-1.05) {\small \faSearch};
 \end{tikzpicture}
  \includegraphics[width=7cm,trim={0.5cm, 1cm, 0.8cm, 2cm},clip]
   {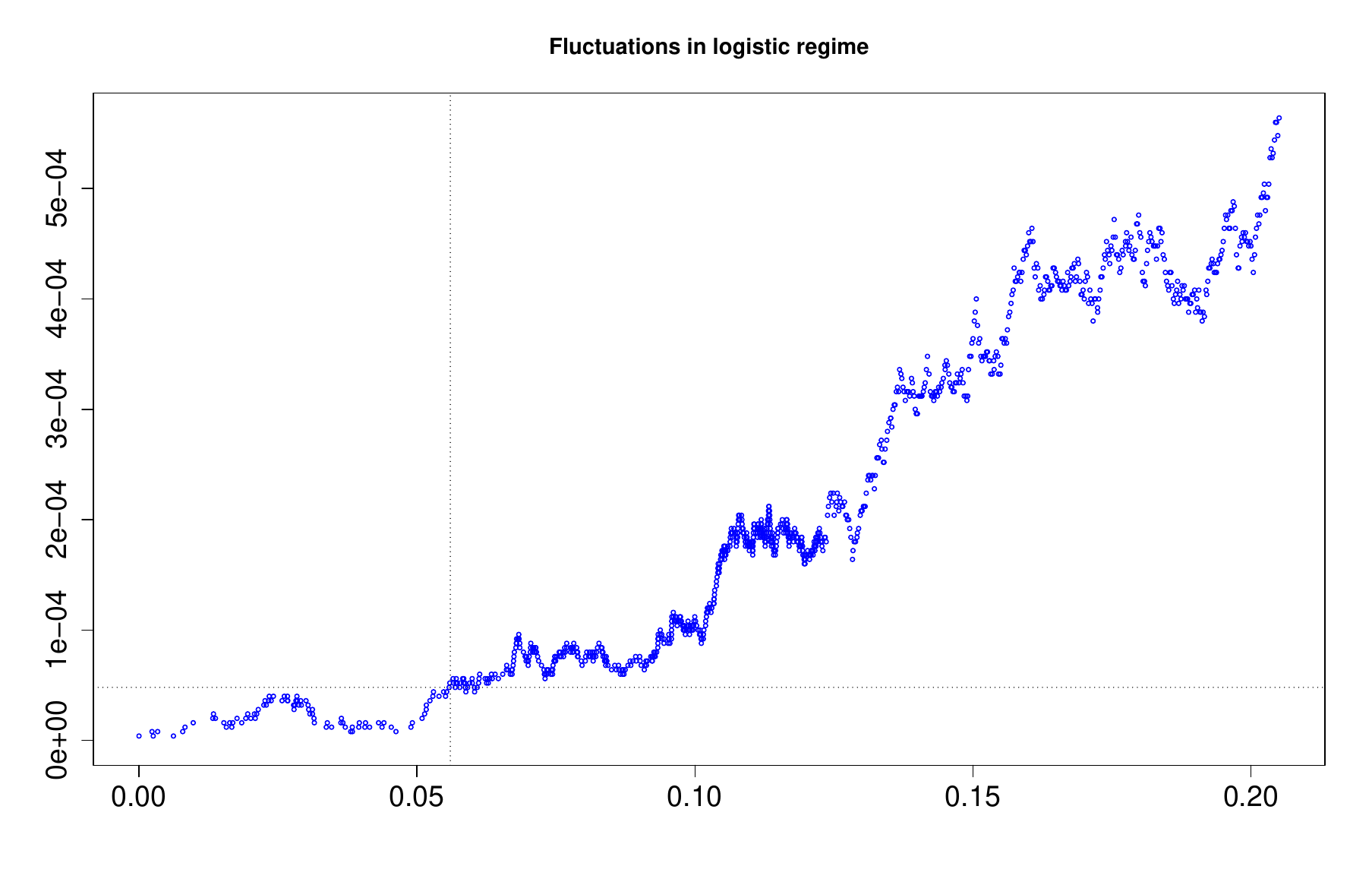}
  \phantom{!}
  \includegraphics[width=7cm,trim={0.5cm, 1cm, 0.8cm, 2cm},clip]
   {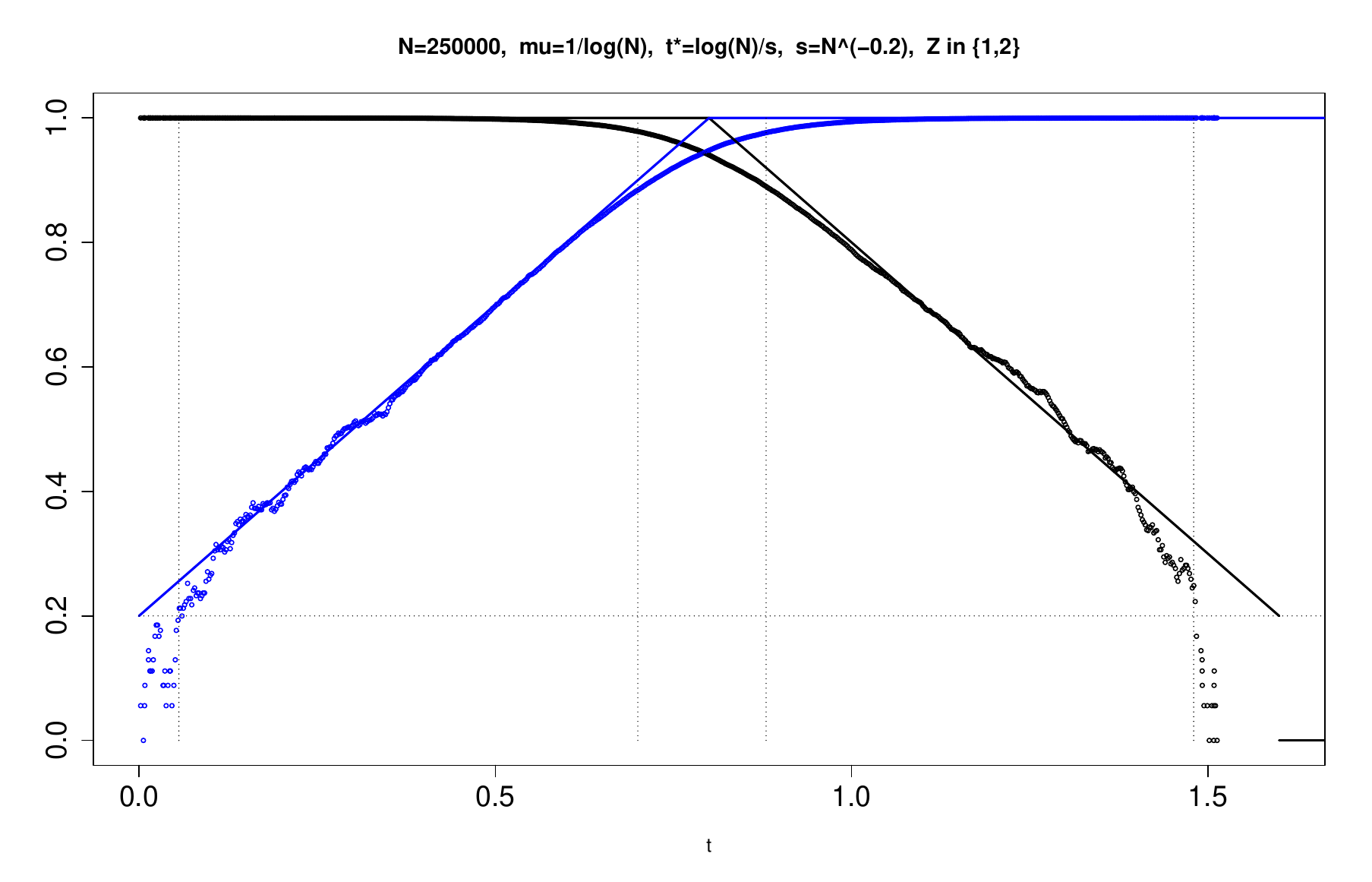}
 \begin{tikzpicture}
    \node at (0,0) {\includegraphics[width=7cm,trim={0.5cm, 1.5cm, 0.8cm, 2cm},clip]{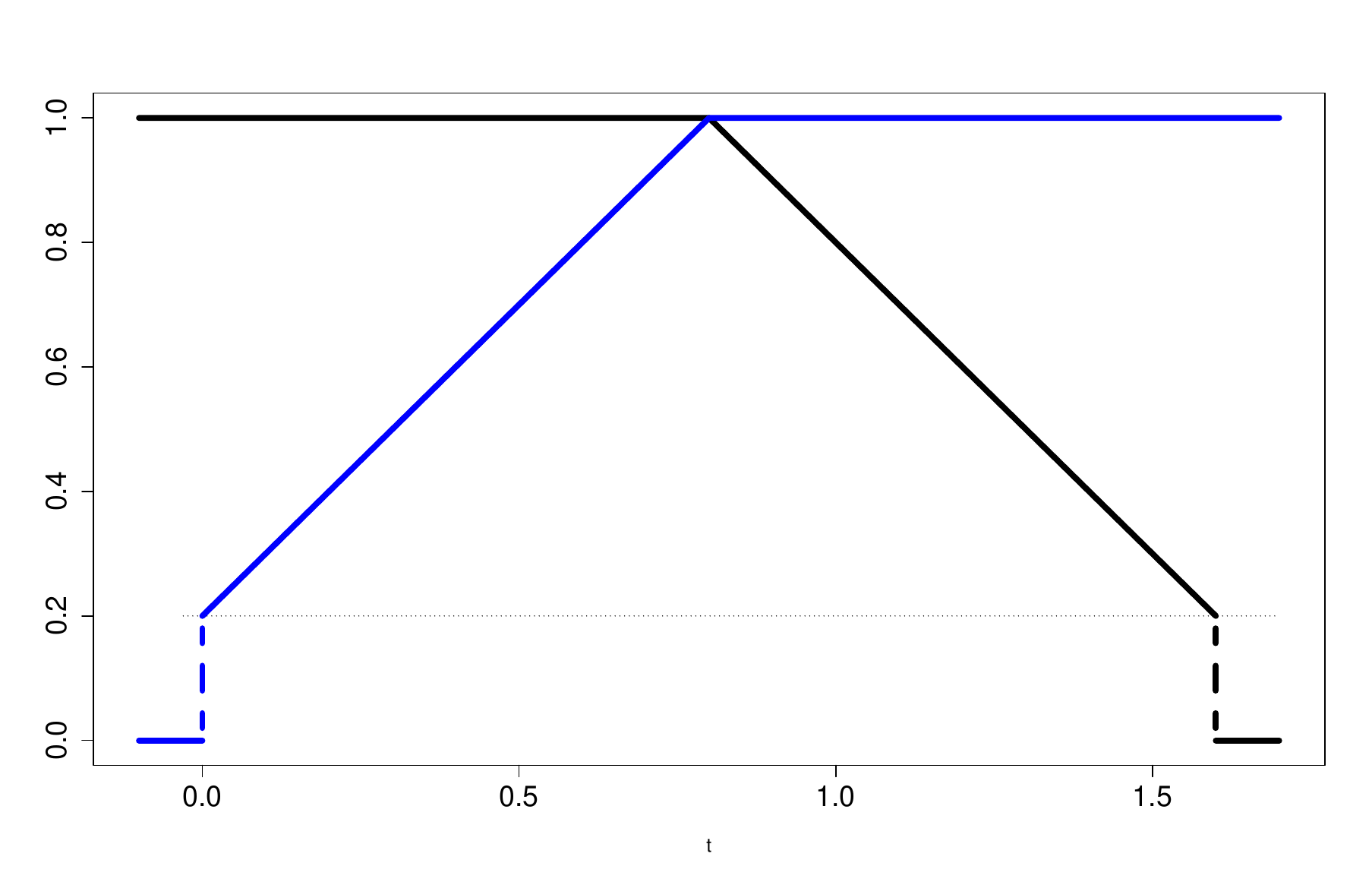}};
    \node[right,black] at (-2.75,1.5) {\small $h_0$};
    \node[right,black] at (0.85,1.2) {\small slope $-a$};
    \node[right,blue] at (-2.75,-0.3) {\small $h_1$};
    \node[right,blue] at (-1.25,0.3) {\small slope $a$};
    \node[black] at (2.87,-0.2) {\small $\tfrac{2(1-b)}{a}$};
    \draw[->,black] (2.87,-0.55) -- (2.87,-0.8);
    \node[black] at (-2.8,-0.825) {\small $b$};
 \end{tikzpicture}
   \caption{\label{fig:single_invader}
     This figure depicts a simulation of a selective sweep curve in the Moran model with moderate selection of population size $N=250\,000$ on the timescale $\frac1s\log N$ in comparison to the limiting \emph{house}. One newcomer (blue) of selective advantage $s=N^{-0.2}$ invades the former resident population (black) and fixates eventually.
     Top left: $\frac1N(X_0^N,X_1^N)$ until time of fixation;
     top right: a zoom into the time interval $[0,0.2] $, showing $\frac1NX_1^N$  and its fluctuations;  horizontal line marks $1/sN$, vertical line marks transition from phase 1 to 2;
     bottom left: dots show the \emph{logarithmic frequencies} $(H_0^N,H_1^N)$ (cf.\ Definition~\ref{defHN}), straight lines indicate scaling limit, horizontal line marks $\log_N(1/s)$, vertical lines illustrate the transitions between the 5 phases; all these 3 panels use the same stochastic realization.
     Bottom right: detailed depiction of the scaling limit, the house, via $(h_0,h_1)$ defined in \eqref{def-house} with $a=1, b=0.2$ to fit the choice of $s$, i.e.\ $s=aN^{-b}$.   
}
\end{figure}

It has been observed already  in pioneering papers  on selective sweeps (\cite{smith1974hitch,KHL89}) that the random fluctuations in the ``early phase of the sweep'' require an approach that is different from the law-of-large-numbers approximation leading to~\eqref{logistic}. 
It is, however, readily checked that $$\log_N(Nx_{1/N}(t\log N))\to st \quad\mbox{as } N\to \infty \quad\mbox{for }t \in (0,1/s),$$ and it is this {\em logarithmic transform} under which in the strong selection regime a law-of-large-numbers approximation of $X^N_1$ turns out to be robust over the entire duration of the sweep.
Indeed, for the {\em logarithmic frequencies} $H_0^N$ and $H_1^N$ of the resident and the mutant population (as specified in Definition~\ref{defHN} with $s= a\varphi_N$ and $\varphi_N\equiv1$) it was proved in~\cite[Corollary~4.8]{HGSTW24}  that, conditional on the fixation events~$\{\mathcal T^N_1 < \infty\}$, the sequence of $[0,1]^2$-valued processes $(H_0^N, H_1^N)$ converges in probability uniformly to the $[0,1]^2$-valued function  $t \mapsto (1\wedge(2-st)\vee 0, st\wedge 1)$.
This is in line with the fact that,  for $s$ not scaling with $N$, the duration of a selective sweep(in the original timescale)  is asymptotically equivalent to $2\frac {\log N}s$ as $N\to \infty$, cf.\ the discussion in the third paragraph of the present section.

\subsubsection*{Moderate and quasi-strong selection}
This is the regime in which $1/N \ll s \ll 1$. As explained above, heuristic arguments then suggest $2\frac {\log (sN)}s$ as a better approximation  of the sweep duration, and also predict that the mutant's Malthusian growth starts instantaneously from size approximately $\tfrac 1{s}$, whereas the original resident's Malthusian decay ends abruptly at size approximately $\frac 1s$. 
Our main result, Theorem~\ref{thm:house}, confirms this intuition and turns some of the (partially heuristic) statements from~\cite{barton1998effect} and~\cite{baake2019modelling}   into a limit theorem that extends~\cite[Corollary 4.8]{HGSTW24}: while for strong selection the sequence  $\min(H_0^N, H_1^N)$ converges as $N\to \infty$ to a ``tent'', this same sequence converges to a ``house'' in the regime of moderate selection, see Figure~\ref{fig:single_invader} (bottom panels). 

\subsubsection*{Recent related work}

\begin{figure}
\includegraphics[width=4.7cm,trim={0.5cm, 1cm, 0.8cm, 2cm},clip]
{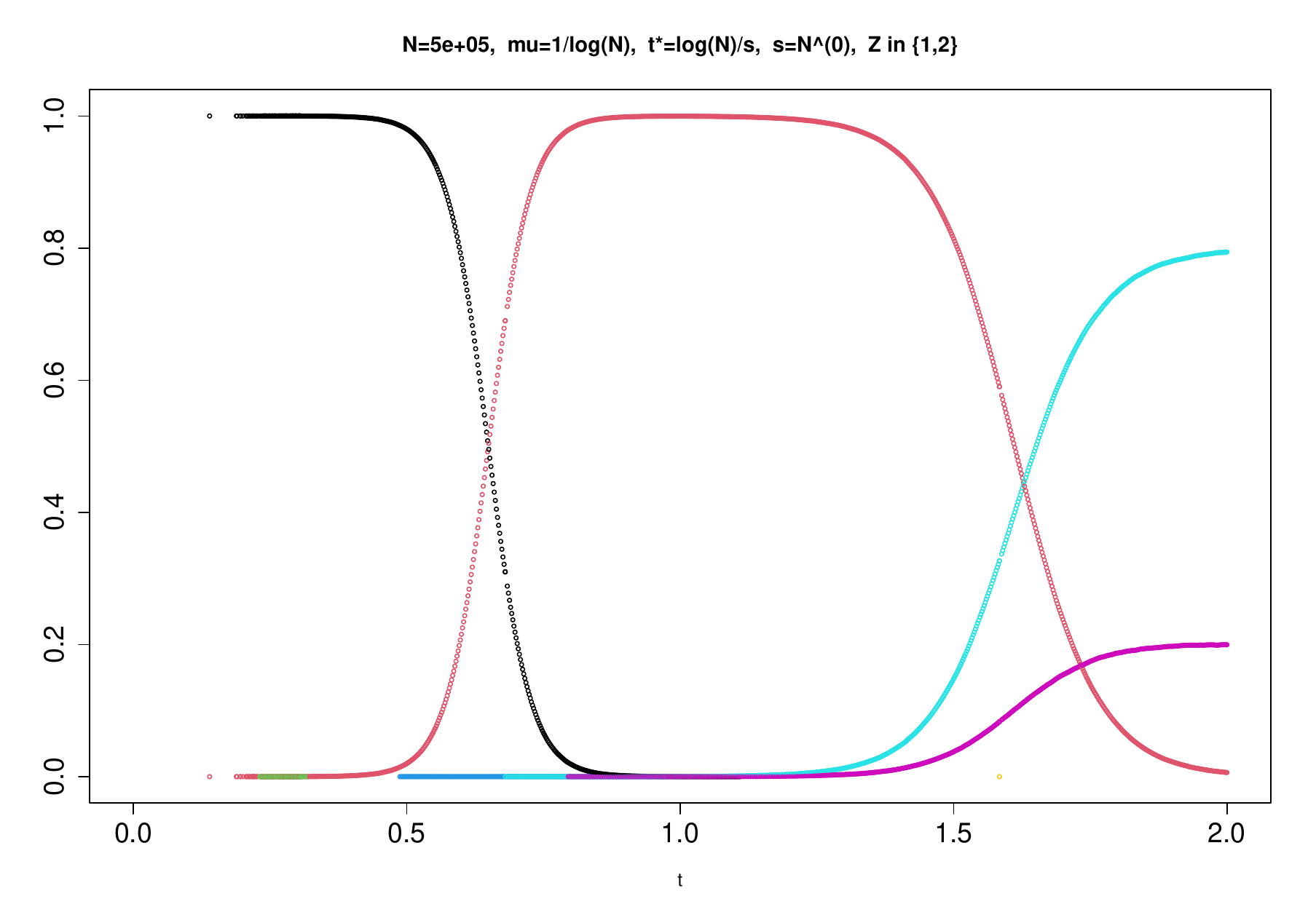}
\includegraphics[width=4.7cm,trim={0.5cm, 1cm, 0.8cm, 2cm},clip]
{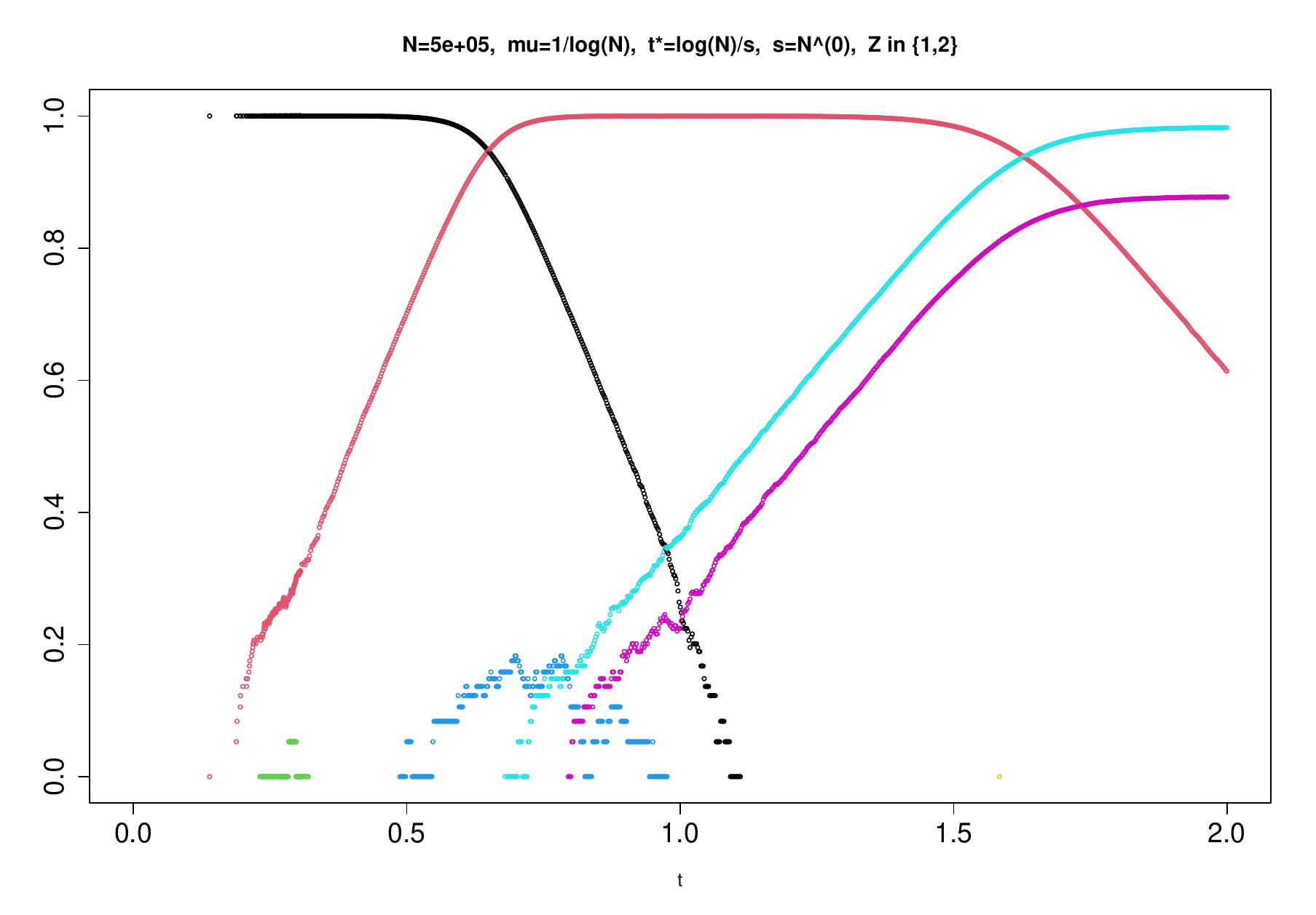}
\includegraphics[width=4.7cm,trim={0.5cm, 1cm, 0.8cm, 2cm},clip]
{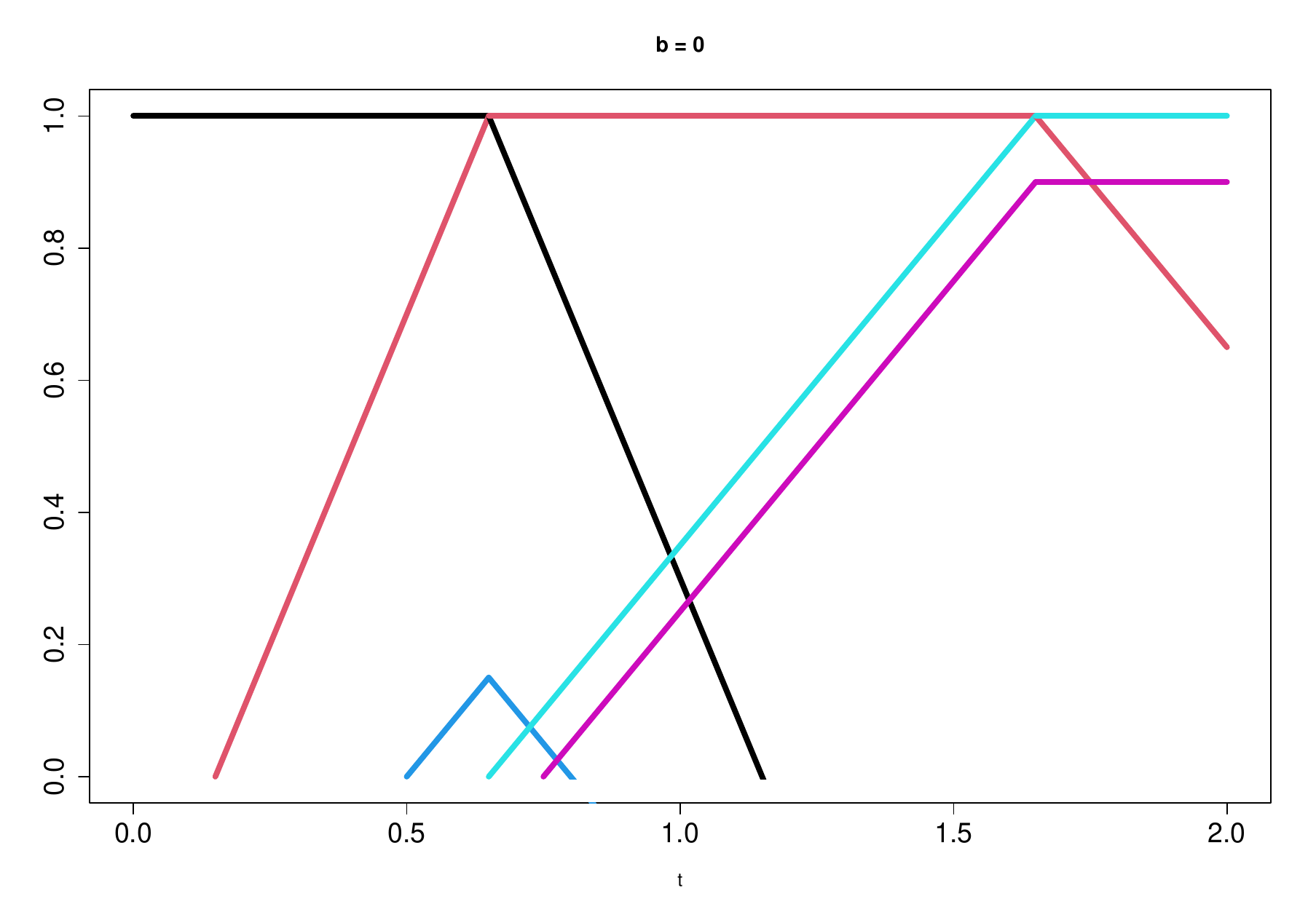} 
\includegraphics[width=4.7cm,trim={0.5cm, 1cm, 0.8cm, 2cm},clip]
{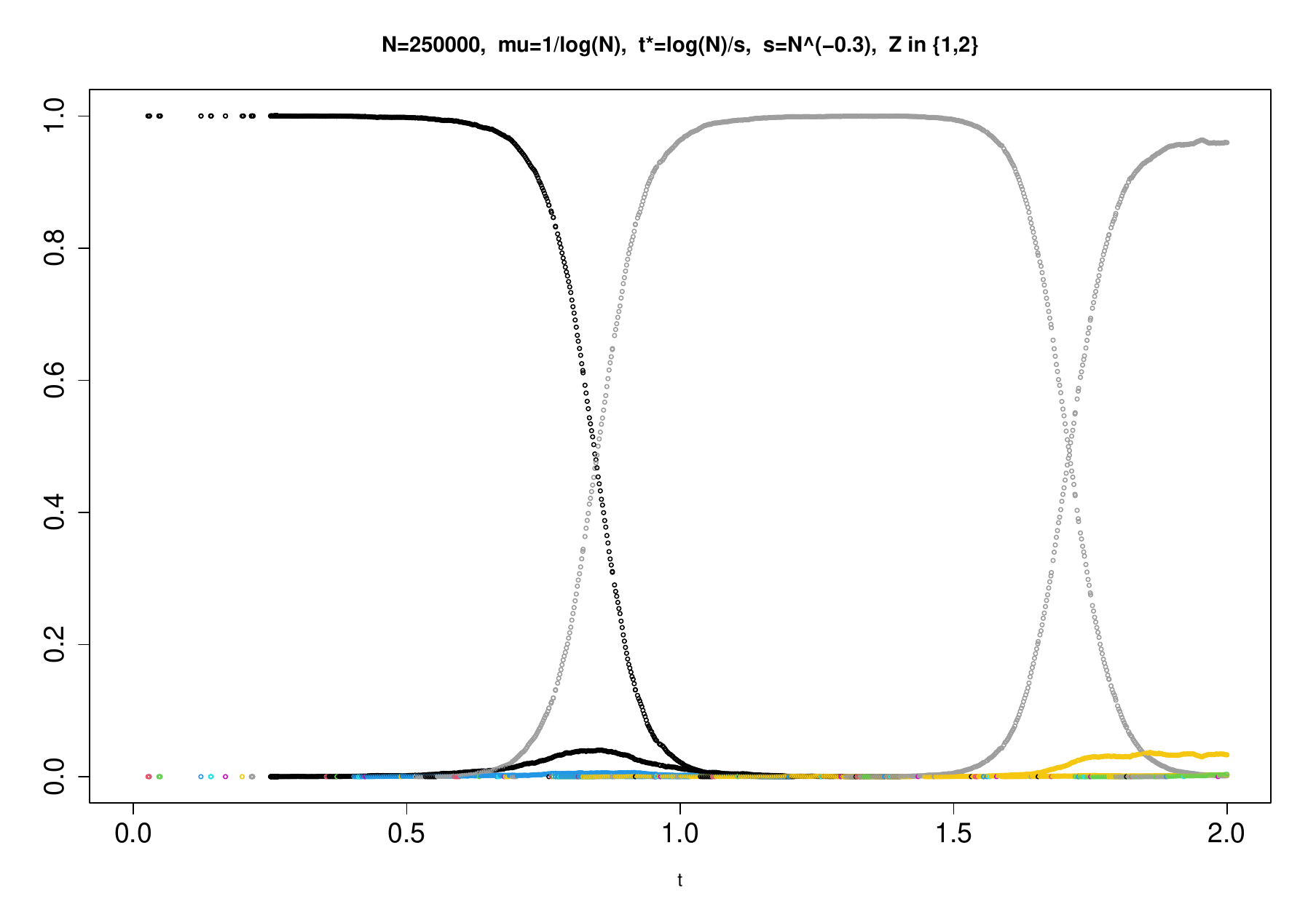}
\includegraphics[width=4.7cm,trim={0.5cm, 1cm, 0.8cm, 2cm},clip]
{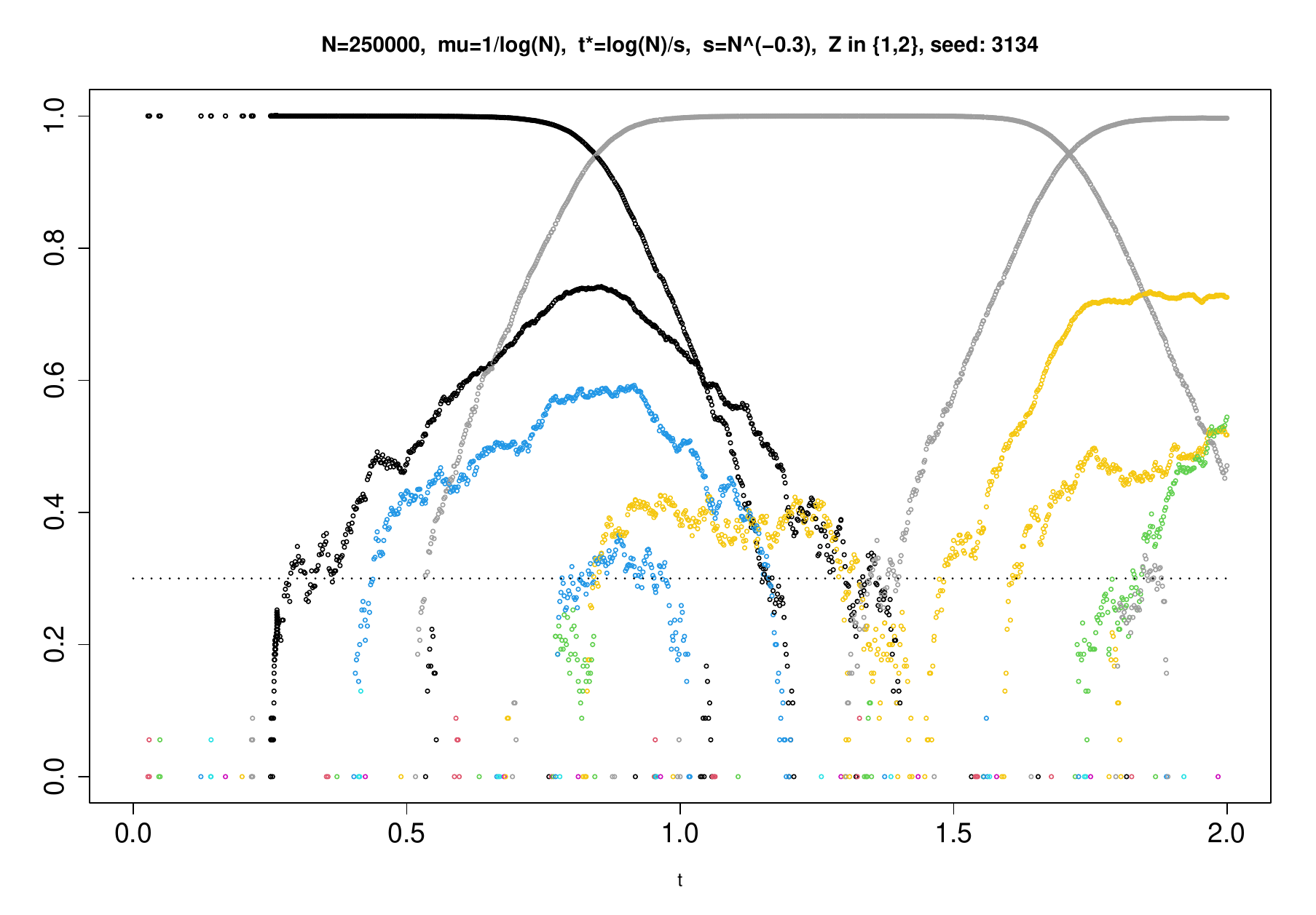}
\includegraphics[width=4.7cm,trim={0.5cm, 1cm, 0.8cm, 2cm},clip]
{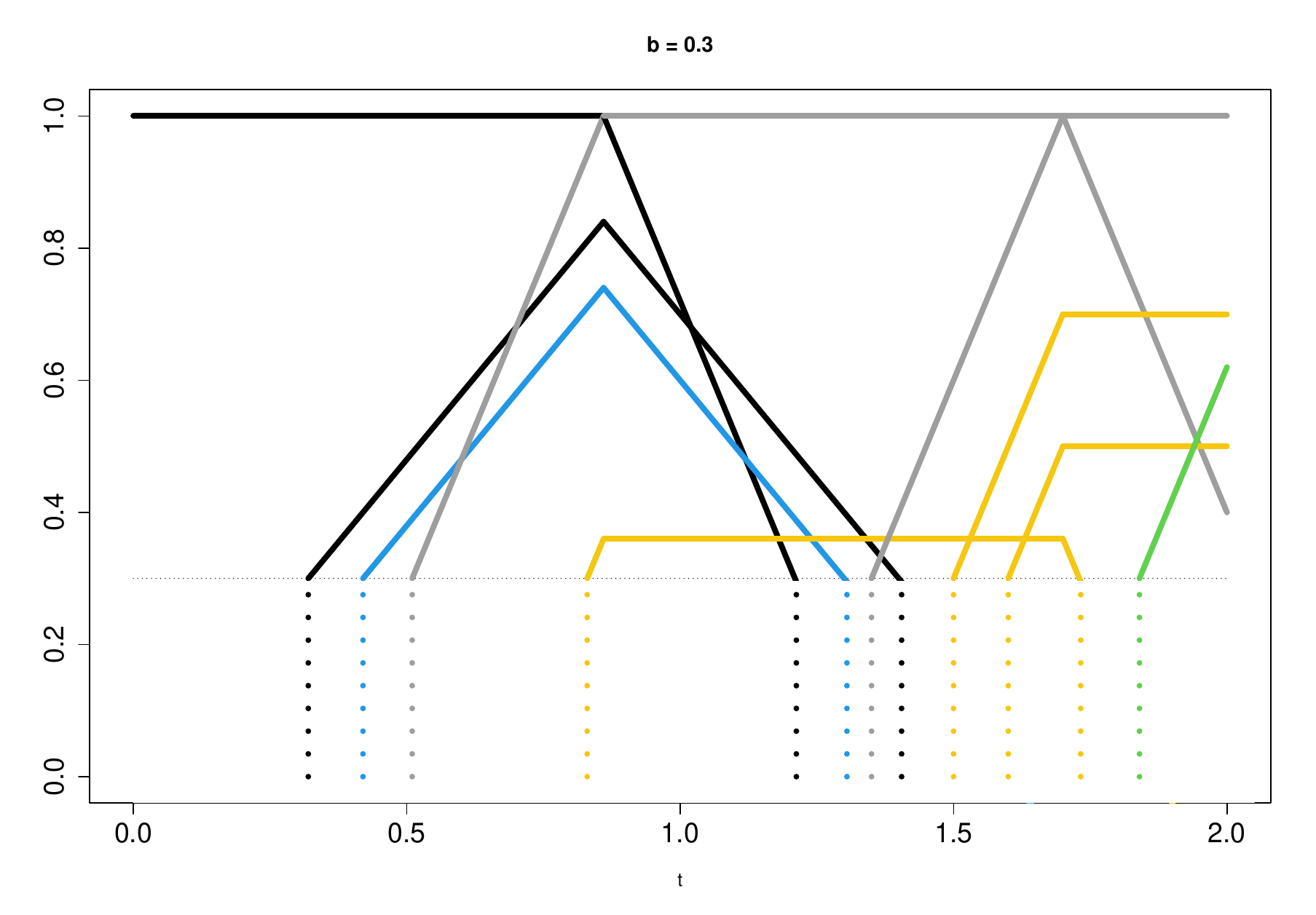} 
\caption{\label{fig:storyline}
This figure depicts simulations of a Moran model with mutation and selection
in the Gerrish--Lenski regime (cf.\ the paragraph on recent related work in Section~\ref{secIntro}, and Section~\ref{CImoderate}), where fitness increments ($s$) are chosen uniformly at random from $\{\varphi_N,2\varphi_N\}$.
Top: strong selection (i.e.\ $\varphi_N\equiv1$), $N=500\,000$; bottom: moderate selection (specifically $\varphi_N=N^{-0.3}$), $N=250\,000$.
Left: logistic sweeps obtained by $\frac1NX^N$; middle: logarithmic scaling, $H^N$, showing piecewise linear behaviour; right: corresponding scaling limits.
}
\end{figure}

For recurrent {\em strongly beneficial} mutations, a scaling limit for the logarithmic frequencies of the mutant families was obtained in~\cite{HGSTW24} for a population-wide mutation rate $m$ of order $\tfrac 1{\log N}$. In this so-called Gerrish--Lenski regime the expected number of mutant families that arrive during a selective sweep and survive the genetic drift remains away from $0$ and $\infty$ as $N\to \infty$. These families then perform interactions, called {\em clonal interference}, illustrated in Figure~\ref{fig:storyline}~(top middle). The scaling limit of the logarithmic frequencies is described by a system of {\em Poissonian interacting trajectories} (Figure~\ref{fig:storyline}, top right). We anticipate that Theorem~\ref{thm:house} will be a central building block in an extension of the results  of~\cite{HGSTW24} to the regime of moderate selection (cf.\ Figure~\ref{fig:storyline}, bottom). This will be further discussed in Section~\ref{CImoderate}.

The work of~\cite{udomchatpitak2025accumulation} considers mutations with moderate selective advantage $s$ that arrive in the population at a rate $m \ll \tfrac 1{\log N}$. The authors show that this leads, as $N\to \infty$, to a Poisson process of selective sweeps on the  $\tfrac 1{sm}$-timescale, with the  duration of any such sweep being asymptotically negligible on that timescale. Even though in that regime there may be a large number of small mutant families present during a sweep that do not escape the genetic drift, we conjecture that their interaction with the  sweep is negligible. Thus, ``zooming in'' around a sweep (by locally expanding the time axis from $\tfrac 1{sm}$ to $\tfrac{\log N}{s}$) will, under a logarithmic transform as in Definition~\ref{defHN}, lead to the same limit of the sweep curve as in our Theorem~\ref{thm:house}, with distinct sweeps, or ``houses'', then separated in time.

The transform $X\to \log_N(X)$
which is central in our work, is referred to as {\em stochastic Hopf--Cole transform}  in ~\cite{desmarais2025k}, and is used to capture (there in a spatial, discrete-time model) the large population limit of time evolution of the fitness distribution in terms of powers of $N$.

\subsubsection*{Structure of the paper} After presenting the main results in Section~\ref{sec-results} we analyse in Sections~\ref{rescrb} and~\ref{pfsecGW} the asymptotic shape of the three middle phases of the sweep, as well as the duration of phases 1 and 5. This prepares for the proof of part I of Theorem~\ref{thm:house}, which is completed in Section~\ref{pfpart1old}. In Section~\ref{secpfM1top} we zoom into phases 1 and 5 of the sweep, and complete the proof of part~II of Theorem~\ref{thm:house} based on fluctuation results for discrete-time embeddings that are provided in Section~\ref{pfpropM1}. In Section~\ref{sec:discu} we give an informal summary of our main results, explain some intuition behind the proof techniques, and give an outlook on work in progress and open problems.

\section{Model and main result}\label{sec-results}

\subsection{Two-type Moran model with selection, and a logarithmic scaling of sweeps}\label{sec:model}
In the following we write the selection coefficient in the form $s_N=a\varphi_N$. Our motivation for this is to separate the effects of the scaling factor $\varphi_N$ and the pre-factor~$a$ that may vary from mutant to mutant in a situation like the one considered in~\cite{HGSTW24} for strong selection, i.e.\ for $\varphi_N \equiv 1$. See also Section~\ref{CImoderate}. 
\begin{definition}\label{def:Moran-Model}
Fix $a, \varphi_N > 0$ and let $X^N=(X_0^N(t),X_1^N(t), t \geq 0)$ be the counting process of the resident and mutant in the Moran model, that is we have $X_0^N(t)+ X_1^N(t)=N$, $\forall t \geq 0$ and with transition rates from $(N-k,k)$ given by
\begin{align}\begin{split}\label{Morandyn}
    (N-k,k) \to (N-k-1,k+1) &\quad \text{ at rate } \quad (1+a\varphi_N) \frac{k(N-k)}N\\
    (N-k,k) \to (N-k+1,k-1) &\quad \text{ at rate } \quad   \phantom{(1+a \varphi_N)}\frac{k(N-k)}N.
    \end{split}
\end{align}
\end{definition}

\begin{remark}\label{remarkprel}
  \begin{enumerate}[a)]
    \item Note that the process $X_1^N$ arises as a time-changed supercritical binary branching process, whose branching rate is slowed down by the factor $\frac {N-k}N$ when it is in state $k$. Likewise, $X_0^N$ arises as a time-changed subcritical binary branching process. This motivates a coupling to such processes that  we will apply in later proofs.
    \item Changing time by the state-dependent factor
    $\tfrac{k(N-k)}N$  takes $X_1^N$ into a random walk whose ratio of upward to downward jump probabilities equals $1+a\varphi_N$.  This leads to the well-known formula for the fixation probability
    \[ \P(\mathcal T_1^N<\infty| X_1^N(0) = 1) \sim \tfrac {a\varphi_N}{1+a\varphi_N} \mbox{ as } N\to \infty. \numberthis\label{fixprob} \]
  \end{enumerate}
\end{remark}

Throughout the present work we will assume that 
\begin{equation}\label{propphiN}
 \varphi_N \le 1 \quad\mbox{and} \quad   b_N := -\log_N(\varphi_N)\to b
 \quad \mbox{as } N\to \infty,
\end{equation}
for some $b \in [0,1)$. This includes the case $\varphi_N \equiv 1$, which was in the focus of~\cite{HGSTW24}. Note that~\eqref{propphiN} implies the existence of an $\eta < 1$ such that
\begin{equation}\label{lowerboundphi}
    \varphi_N \gg N^{-\eta} \quad \mbox{as } N\to \infty.
\end{equation}

\begin{definition}[Selection regimes]\label{def:moderate-selection}  We will call $\varphi_N$ the {\em scaling factor}, and $b$ the {\em exponent} of the selection.
    We say that selection is \emph{moderate}
    if $b\in(0,1)$, {\em quasi-strong} if $b=0$ and $\varphi_N \to 0$, and {\em strong} if $\varphi_N\equiv 1$. 
\end{definition}
  The case $s_N= \frac aN$, corresponding to $b=1$, is commonly referred to as \emph{weak selection}. In this regime, diffusion processes arise as scaling limits of $(\frac1NX^N(tN))_{t \ge 0}$, i.e.\ in the {\em evolutionary timescale} whose unit is $N$ generations (see \cite{ewens2004mathematical}, \cite{ethier2009markov}). Selective sweeps (and their impact on genetic hitchhiking) have been studied in the diffusion approximation by letting $a\to \infty$, see~\cite{EPW06} and references therein. 

Prototypical choices for the scaling factors with  moderate and quasi-strong selection are $\varphi_N=N^{-b}$  and $\varphi_N = \frac 1{\log N}$.
With the stopping times $\mathcal T_\varepsilon^N$ defined as in~\eqref{deftaueps}, the same arguments as the ones cited there show that, on the timescale $\varphi_N^{-1}$, the frequencies $\frac1NX^N_1(t)$, when conditioned upon survival of $X^N_1$, again converge to the logistic curve~\eqref{logcurve}, now with $a$ in place of $s$
(cf.\ Figure~\ref{fig:single_invader}, top left).
However, we focus on the exponential growth and decline phases that live on a longer timescale by the factor $\log N$. Scaling the population sizes by $\log_N$ turns the exponential growth into a linear one, which motivates the following definition of our main object of interest.

\begin{definition}\label{defHN}
    The \emph{logarithmic population sizes} $H^N=(H^N_0,H^N_1)$ are defined by
  \begin{equation}\label{logscaling}
      H^N_i(t)
        = \log_N^+\big(X^N_i(\tfrac{t\log N}{\varphi_N})\big),
        \ i\in\{0,1\},
  \end{equation}
    where $\log_N^+(\cdot):= \max\{\log_N(\cdot), 0\}$ denotes the positive part of the logarithm to the base $N$.

    Throughout we will assume that $(X_0, X_1)$ satisfies for $t\ge 0$ the Moran dynamics~\eqref{Morandyn} with $a>0$ and starting condition $(X_0(0), X_1(0)) = (N-1, 1)$ where selection is moderate  or quasi-strong in the sense of Definition~\ref{def:moderate-selection}.
    With a view towards part II of Theorem~\ref{thm:house} we extend $X^N$ to the negative time axis by putting $(X_0(t), X_1(t)) = (N,0)$ (and consequently $(H_0^N(t), H_1^N(t)) = (1,0)$) for $t<0$.
\end{definition}

\subsection{Convergence to the shape of a house}\label{mainresults}

The piecewise linear trajectories appearing in Figure~\ref{fig:single_invader} (bottom panels) are defined as
\begin{align}\label{def-house}
  h_0(t)
  &:= \begin{cases}
        1 & \text{for } t < \tfrac{1-b}a,\\
        2-(b+at)
          & \text{for }\tfrac{1-b}a\leq t<\tfrac{2(1-b)}a,\\
        0 & \text{for }t\geq \tfrac{2(1-b)}a.
      \end{cases}
  &\text{and}&&
  h_1(t)
  &:= \begin{cases}
       0 & \text{for } t < 0  \\b+at & \text{for }0\le t < \tfrac{1-b}a,\\
        1 & \text{for }t\geq\tfrac{1-b}a.
      \end{cases}
\end{align}
For $b=0$ (i.e.\ both for strong and for quasi-strong selection) the functions $h_i$ are continuous, and the graph of $\min(h_0,h_1)$ looks like a ``tent''. For $b \in (0,1)$, the functions $h_1$ and $h_0$ have a jump of height~$b$ in their respective jump times $t_1:= 0$ and $t_0:= \tfrac {2(1-b)}a$, so that the graph of $\min(h_0,h_1)$  takes the form of a ``house''. In any case, the functions $h_i$ are elements of
\begin{equation}\label{defD1}
    \mathcal D:=\{f:\R\to [0,1] \mid f \mbox{ is right continuous and has left limits}\},
\end{equation}
and the $H_i^N$ can be seen as $\mathcal D$-valued random variables.
Our main result states that,  for $i\in \{0,1\}$, the sequence $H_i^N$ converges, as $N\to \infty$, in probability to $h_i$
in the Skorokhod $M_1$~topology, one of the classical topologies
on $\mathcal D$. In Section~\ref{metricM1} we will recall a metric that generates the $M_1$~topology; 
for more background on the Skorokhod topologies see e.g.~\cite{K23} and references therein.
 Roughly stated, Theorem~\ref{thm:house} says that the scaled logarithmic frequency of the mutant (resp.\ the resident) population converges uniformly in probability to $h_1$ (resp. $h_0$) on the time segment where the limiting path is continuous, and  climbs over the  jump height of its limiting path in a nearly monotone way (i.e.\ without substantial backtracking). 
See Figure~\ref{fig:single_invader} (bottom left) for an illustration.
\begin{theorem}\label{thm:house}
    Let $(H^N_0,H^N_1)$ be as in Definition~\ref{defHN}.
    Then, conditional on the event of fixation $\{\mathcal T_1^N < \infty\}$, for $i\in\{0,1\}$, 
    \begin{itemize}
      \item[I.1)] the rescaled sequence of fixation times
        $\sigma_{\rm fix}^N:=\frac{\mathcal T_1^N}{ \varphi_N^{-1}\log N}$ converges, as $N \to \infty$, in probability to~$\frac{2(1-b)}{a}$,
      \item[I.2)] for any $\varepsilon >0$,  the sequence $H_i^N$ converges as $N\to \infty$ in probability to $h_i$, with respect to the uniform distance when restricted to the set $D^\eps_i$ defined as
  \begin{equation}\label{defsetsD}
  D_1^\eps:= \R\setminus [0, \eps), \qquad D_0^\eps:= \R \setminus\Big(\tfrac{2(1-b)}a-\varepsilon, \tfrac{2(1-b)}a+\varepsilon \Big),  
  \end{equation}
     \item[I.3)] in the case $b=0$ the uniform convergence in probability of $H_i^N$ to $h_i$  holds on $\R$,
     \medskip
     \item[II)] the sequence $(H_i^N)$, viewed as random elements of $\mathcal D$,
       converges, as $N\to \infty$ in probability to~$h_i$. Hereby, $\mathcal D$ is equipped with the Skorokhod $M_1$~topology.
     \end{itemize}
\end{theorem}
\begin{remark}\label{rem26}
\begin{enumerate}[a)]
    \item In the special case $\varphi_N \equiv 1$ (i.e. for strong selection), the assertions of Theorem~\ref{thm:house} follow from~\cite[Corollary~4.8]{HGSTW24}. 
    \item Part I.2) of Theorem~\ref{thm:house} implies that the increase of $H_1^N$ 
from height $0$ to height $b$ occurs quickly after time~$0$, and that the decrease of $H_0^N$ from $b$ to $0$
occurs quickly around time 
$\tfrac {2(1-b)}a$. This statement on the slimness of the house's walls will be sharpened in Section~\ref{pfstrat} based on Propositions~\ref{Zph1} and~\ref{Zph5}, as one of the ingredients of the proof of part II of Theorem~\ref{thm:house}.
    
\item Among the Skorokhod topologies, the one that is probably the best known and most widely used  is $J_1$. Since the maximal jump height of $H_i^N$ converges to $0$, while $h_i$ has a jump of positive height for $b >0$, one can easily see that $H_i^N$ cannot converge to $h_i$ in the $J_1$~topology. However, Proposition~\ref{propeastwest} will assert that the increase of $H_1^N$ (decrease of $H_0^N$) mentioned in item b) happens fairly regularly, 
with asymptotically negligible fluctuations in between. This will be key for the proof of the \mbox{$M_1$ convergence} of $H_i^N$ to~$h_i$.  At this point we cannot resist quoting a sentence from the preamble of \cite{Whitt}: {\em Thus, we would be so bold as to suggest that, if only
one topology on the function space $\mathcal D$ is to be considered, then it should be the $M_1$~topology.}
\end{enumerate}
\end{remark}

\section{Proof of Theorem~\ref{thm:house} part  I}\label{pfpart1}
In this section we prepare for the proof of part I of Theorem~\ref{thm:house}, first describing the strategy and then stating a series of auxiliary results whose proofs will be given in  Section~\ref{pfsecGW}.
\subsection{Strategy of the proof}\label{pfstrat}
As indicated in Remark~\ref{remarkprel} a),  representations of $X_1^N$ and of $X_0^N$ as time-changes of linear birth-death processes (or binary branching Galton--Watson  processes) $Z_1^N$ and $Z_0^N$ will be helpful. Here,
\begin{align}\label{Zrates}\begin{split}Z_1^N &\mbox{ jumps from } k \mbox{ to } \begin{cases} k+1 \mbox{ at rate } (1+a\varphi_N)k\\
k-1 \mbox{ at rate } k,\end{cases}
\\
Z_0^N &\mbox{ jumps from } k \mbox{ to } \begin{cases} k+1 \mbox{ at rate } k\\
k-1 \mbox{ at rate } (1+a\varphi_N)k.\end{cases}
\end{split}
\end{align}
{According to the assumptions in Theorem~\ref{thm:house} the process $X_1^N$ starts in $1$ and is conditioned to reach~$N$ before $0$. We will therefore work in Proposition~\ref{Zph1} (which concerns~phase 1 of the sweep) with the process $Z_1^N$ conditioned on survival. This conditioned process has jump rates different from those stated in~\eqref{Zrates}. We will not need the explicit form of these rates but will make use of the fact that  the time-change between $Z_1^N$ and $X_1^N$ that is described in Remark~\ref{remarkprel} a) is not affected by the conditioning.}

{As long as $X_1^N \ll N$,} the time-change factor is {close to $1$,}  
so during that period $X_1^N$ will be quite accurately approximated by  
$Z_1^N$.
We will see that {$Z_1^N$}, when started in $Z_1^N(0)=1$ and conditioned to survival, reaches with high probability (w.h.p.) size $\tfrac {\log N}{\varphi_N}$ in a time which is of smaller order than $\tfrac {\log N}{\varphi_N}$. (Thus the quantity $\tfrac {\log N}{\varphi_N}$ takes here the role both of a critical height of $Z_1^N$ and a time-scaling factor, where in the critical height the $\log$ could be replaced by any  slowly varying function that increases to $\infty$.)
This can be seen as phase 1 of the sweep, turning out to be asymptotically negligible on the ``sweep timescale'' whose unit we take as $\tfrac {\log N}{\varphi_N}$.

We will then prove that, with $Z_1^N$ started in state $\tfrac {\lfloor \log N\rfloor }{\varphi_N}$, the sequence $t \mapsto \log^{\textcolor{lightgray}{+}}_N(Z_1^N(\tfrac {t\log N}{\varphi_N}) 
$
converges as $N\to \infty$ {for $t \in [0, \frac{1-b}a]$}   uniformly in probability to the function $t\mapsto b+at$ ; this will take care of phase~2 of the sweep. 

Now we turn to the period between the hitting times $\mathcal T_{1/\sqrt{\log N}}^N$ and $\mathcal T_{1-1/{\sqrt{\log N}}}^N$, which can be seen as phase 3 of the sweep. During that period the time change factor obeys
\begin{equation}\label{lowerbound}
    1-\tfrac{X_1^N}N \ge \tfrac1{{\sqrt{\log N}}}.
\end{equation}
We will prove that the time it takes the process $Z_1^N$ to reach $N(1-\tfrac 1{{\sqrt{\log N}}})$ when started in $\tfrac N {\sqrt{\log N}}$ is  $o\big(\tfrac {\sqrt{\log N}}{\varphi_N}\big) $  w.h.p., 
so,  with the lower bound~\eqref{lowerbound} on the time-change factor being taken into account, we will be able to conclude that the duration of phase 3 of the sweep is negligible compared to $ \frac{\log N}{\varphi_N}$. 

As soon as the invader's size $X_1^N$ has reached $\lfloor N(1-\tfrac 1{{\sqrt{\log N}}})\rfloor$ (which happens at time $\mathcal T_{1-1/{\sqrt{\log N}}}^N$), the original resident's size $X_0^N$ has come down to size 
{$N- \lfloor N(1-\tfrac 1{\sqrt{\log N}})\rfloor\sim \tfrac N{\sqrt{\log N}}$}. 
Comparing $X_0^N$ to a driftless random walk we will show that w.h.p.\ after this time $X_0^N$ will never again reach the size $\tfrac N{(\log N)^{1/3}}$ before it gets extinct.

In the remaining period of the sweep we work with the subcritical binary branching Galton--Watson process~$Z^N_0$ whose  jump rates are given in~\eqref{Zrates}. The process $Z^N_0$ is connected to $X_0^N$ by the time-change factor $1-\tfrac{X_0^N}N$. As long as $X_0^N$ is not above  size $\tfrac N{(\log N)^{1/3}}$, this time-change is between $1$ and $1-\frac 1{(\log N)^{1/3}}$, so in phases 4 and 5 we can safely approximate $X_0^N$ by~$Z^N_0$.

Two features of $Z_0^N$ are now crucial: a ``Malthusian'' (law-of-large-numbers) behaviour until (slightly above) the size $\frac 1{\varphi_N}$ (in phase 4), followed by  a quick surrender to the fluctuations (in phase~5). More specifically, 

(i) with  {$Z_0^N$} started in {$ \lfloor \tfrac N{\sqrt{\log N}}\rfloor$} and stopped as soon as it comes down to $\tfrac {\log N}{\varphi_N}$, {during this period} the logarithmic transforms $t \mapsto \log^{+}_N(Z_0^N(\frac {t\log N}{\varphi_N}))$ converge    as $N\to \infty$ uniformly in probability to the function $t \mapsto 1-at$,

(ii) when started in $\lfloor \tfrac {\log N}{\varphi_N}\rfloor$, the time to extinction of $Z_0^N$ is of smaller order than $\tfrac {\log N}{\varphi_N}$.

\subsection{The accompanying Galton--Watson processes in phases 1 -- 5} \label{rescrb}In this section we formulate the results on the  Galton--Watson processes $Z_1^N$ and $Z_0^N$ (with jump rates given in~\eqref{Zrates})  that were announced in Section~\ref{pfstrat}. 
As         before, $a$ is a strictly positive number, and the sequence of scaling factors $\varphi_N$ obeys~\eqref{propphiN} with $b \in [0,1)$. We put
\begin{equation}\label{def:TNk}
    T_k^N:=\inf\{t\ge 0\mid Z_1^N(t) =k\}.
\end{equation}
and abbreviate the logarithmic frequencies as
\begin{equation}\label{logscal}Y_i^N(t):= \log_N^+(Z^N_{i} (t \varphi_N^{-1} \log N)), \quad i=0,1.
\end{equation}
Each of the following five propositions refers to a phase of the sweep. 
For example, Proposition~\ref{Zph2} describes the asymptotically linear increase of the logarithmic frequency of the mutant, which, according to Proposition~\ref{Zph1}, starts after a random time delay whose duration is asymptotically negligible compared to
the duration of the sweep.
This random time delay is similar to the one appearing in
~\cite{barbour2015escape}
even though the thrust of that paper is different from ours.
The proofs of the following results  will be given in Section~\ref{pfsecGW}.
\begin{prop}[Phase 1]\label{Zph1}
        Assume that $Z_1^N$ starts in $1$ and is conditioned under survival, i.e.\ under the event $\{T_0^N =\infty\}$. Then
    \begin{equation}      \frac{\varphi_N}{\log N} \,T_{\big\lfloor\tfrac {\log N}{\varphi_N}\big\rfloor}^N \to 0 \,\,\mbox{ in probability  as } N\to \infty.
    \end{equation}
\end{prop}

\begin{prop}[Phase 2]\label{Zph2}
    Assume that $Z_1^N$ starts in $\big\lfloor\tfrac {\log N}{\varphi_N}\big\rfloor$.
    Then for any $T$, $\delta>0$
    \begin{equation}\label{unifph2}
    \lim_{N\to\infty} \P \Big(  \sup_{t \leq T} \Big|Y_1^N(t) - (b+at) \Big| >\delta \Big)
        = 0,    
    \end{equation}
    i.e.\ $\big(Y_1^N(t),\, t\in[0,T]\big)$ converges uniformly in probability to $(b+at)_{t\in[0,T]}$.
\end{prop}

\begin{prop}[Phase 3]\label{Zph3}
  Assume that $Z_1^N$ starts in $\big\lfloor \frac N{\sqrt{\log N}} \big\rfloor$.  Then
  $$
    \frac{\varphi_N}{\sqrt{\log N}}\, T_{\big\lfloor N\big(1-\tfrac 1{\sqrt{\log N}}\big)\big\rfloor}^N
      \to 0 \,\,\mbox{ in probability  as } N\to \infty.
  $$
\end{prop}

\begin{prop}[Phase 4]\label{Zph4}
  Assume that $Z^N_0$ starts in $N- \big\lfloor N(1- \tfrac {1}{\sqrt{\log N}})\big\rfloor$, and define
  \begin{equation}\label{deftheta}
\vartheta_N := \inf\big\{t\ge 0: Z_0^N( t \varphi_N^{-1}\log N ) \le  \lfloor \tfrac{\log N}{\varphi_N} \rfloor \big\}.
\end{equation}
  Then
  \begin{equation}\label{convtheta1}
 \vartheta_N \to \tfrac{1-b}a \quad\mbox{ in probability as } N\to \infty
  \end{equation}
  and
    \begin{equation}\label{eastroof}
    \lim_{N\to\infty}  \P \Big(  \sup_{t \leq \vartheta_N} \Big| Y_0^N(t) - (1-at) \Big| > \frac{\log \log N}{\log N} \Big)=0. 
    \end{equation}
    In particular, for any $u<\frac{1-b}{a}$, $\big(Y_0^N(t) , t \in [0,u] \big)$ converges uniformly in probability to $(1-at)_{t\in [0,u]}$ as $N\to\infty$. 
\end{prop}

\begin{prop}[Phase 5]\label{Zph5}
Assume that $Z^N_0$ starts in $\big\lfloor \tfrac {\log N}{\varphi_N}\big\rfloor$, and denote the extinction time of $Z_0^N$ by $\widetilde T_0^N$. Then
    \begin{equation}      \frac{\varphi_N}{\log N} \,\widetilde T_0^N \to 0 \,\,\mbox{ in probability  as } N\to \infty.
    \end{equation}
\end{prop}

\subsection{Proof of Propositions~\ref{Zph1}--\ref{Zph5}}\label{pfsecGW}

\subsubsection{Proof of Propositions~\ref{Zph1} and~\ref{Zph5}}
We first recall an elementary result (whose proof, building on~\cite[Sec. III.5]{athreya1972branching}, can be found in~\cite[p.~1394]{HJR20}). 
\begin{lemma}\label{lem:basic-props}
    Let $Z=(Z(t), t \geq 0)$ be a continuous-time binary Galton--Watson process with birth rate~$ \lambda$ and death rate $ \mu$  being strictly positive such that $ \lambda \neq  \mu$, and assume $Z(0)=1$. Then for $t\ge0 $
    \begin{equation}
                \P(Z(t)\geq j)
             = \frac{ \lambda- \mu}{ \lambda- \mu e^{-( \lambda- \mu)t}}\cdot\Big(1-\frac{ \lambda- \mu}{ \lambda e^{( \lambda- \mu)t}- \mu}\Big)^{j-1}, \quad j\in \mathbb N.
    \end{equation}
\end{lemma}
\begin{proof}[Proof of Proposition~\ref{Zph1}] We abbreviate 
\begin{equation}\label{defdt}
  t_N:=  \tfrac {\sqrt{\log N}}{\varphi_N}, \quad j_N:= \big\lfloor \tfrac {\log N}{\varphi_N}\big \rfloor.
\end{equation} Since 
$t_N = o\big( \tfrac {\log N}{\varphi_N}\big) $ as $N \to \infty$, for proving the proposition it suffices to show that
\begin{equation}\label{cl1}
\mathbb P_1\Big(T_{j_N}^N \le t_N \,  \Big| \, Z_1^N \mbox{ survives}\Big) \to 1 \quad \mbox {as } N\to \infty.
\end{equation}
To see the validity of ~\eqref{cl1}, we   observe that its l.h.s.\ is bounded from below by
\begin{eqnarray} \notag 
&&\mathbb P_1\Big(Z_1^N(t_N) \ge j_N \,  \Big | \, Z_1^N \mbox{ survives}\Big)
\ge \frac{\sum_{j\ge j_N}\P_1(Z_1^N(t_N)=j \mbox{ and } Z_1^N\mbox { survives} ) }{ \P_1(Z_1^N \mbox { survives})}
 \\ \notag &=&\frac{\sum_{j\ge j_N}\P_1(Z_1^N(t_N)=j) \P_j( Z_1^N\mbox { survives} ) }{ \P_1(Z_1^N \mbox { survives})}\ge \frac{\sum_{j\ge j_N}\P_1(Z_1^N(t_N) =j ) \P_{j_N}(Z_1^N \mbox { survives})}{ \P_1(Z_1^N \mbox { survives})}
 \\ \label{secondfac}
 &\ge& 
\frac{\P_1(Z_1^N(t_N) \ge j_N ) }{ \P_1(Z_1^N(t_N) \ge 1)}(1-(1-\P_1(Z_1^N \mbox {survives}))^{j_N}),
\end{eqnarray}
where  for $j\in\N$, $\P_j$ denotes the probability measure under which $Z_1^N$ starts in the state $j$ at time $0$, and in the equality we used the Markov property at time $t_N$.
\\
Because of Lemma~\ref{lem:basic-props} (with $ \lambda:= 1+a\varphi_N$, $ \mu:= 1$) we have
$$\frac{\P_1(Z_1^N(t_N) \ge j_N) }{ \P_1(Z_1^N(t_N) \ge 1)}\ge \Big(1-\frac{a\varphi_N}{e^{a\varphi_Nt_N}-1}\Big)^{\tfrac{\log N}{\varphi_N}}.$$
This converges to $1$ as $N\to \infty$, because with our choice of $t_N$ we have
$$\frac{\log N}{\varphi_N} \frac{a\varphi_N}{e^{a\varphi_Nt_N}_{}-1}=\frac {a \log N} {e^{a\sqrt{\log N}}_{}-1}\to 0.$$
It is well known that $\P_1(Z_1^N \mbox{ survives}) = \frac{ \lambda- \mu}{ \lambda}=\frac{a \varphi_N}{1+a \varphi_N}\sim a\varphi_N$. Consequently, since $j_N\gg \varphi_N^{-1}$, also the right-hand factor in~\eqref{secondfac} converges to 1 as $N\to \infty$.   This concludes the proof of~\eqref{cl1}.
\end{proof}
\begin{proof}[Proof of Proposition~\ref{Zph5}] We set  $t_N$ and $j_N$ as in~\eqref{defdt}. As in the proof of Proposition~\ref{Zph1}] we note that it suffices to show
\begin{equation}\label{extinct}
\mathbb P_{j_N}(\widetilde T_0^N \le t_N) \to 1 \quad \mbox {as } N\to \infty.
\end{equation}
Because of the branching property the l.h.s.\ of~\eqref{extinct} equals $\mathbb P_1(Z_0^N(t_N) =0)^{j_N}$, which by  Lemma~\ref{lem:basic-props}   (now with $ \lambda:=1$, $ \mu:= 1+a\varphi_N$)  equals
$$\Big(1-\frac{a\varphi_N}{(1+a\varphi_N)e^{a\varphi_Nt_N}-1}\Big)^{j_N} \ge \Big(1-\frac{a\varphi_N}{e^{a\varphi_Nt_N}-1}\Big)^{j_N}.$$
The  r.h.s.\ converges to $1$ as $N\to \infty$, as argued in the proof of Proposition~\ref{Zph1}. This concludes the proof.
\end{proof}
\subsubsection{Proof of Propositions~\ref{Zph2} and~\ref{Zph4}}
We will prepare for the proof of these propositions by providing an extension of~\cite[Lemma A.1]{CMT21}, which includes the near-critical case $\varphi_N \to 0$ and allows
(in the supercritical case) also for an initial condition $Z^N(0) = \lfloor \tfrac{\log N}{\varphi_N} \rfloor$, as it appears in Proposition~\ref{Zph2}. This is covered by the case $\beta = b$
in Lemma~\ref{lem:CMT}\,a). Our proof of that lemma highlights two main ideas contained in the proof of~\cite[Lemma A.1]{CMT21}:
\begin{enumerate}[\quad 1)]
    \item to apply Doob's inequality to the martingale  $Z^N(t) / \E[Z^N(t)]$,
    \smallskip
    \item to use elementary properties of the $\log$-function.
\end{enumerate}
\begin{lemma}\label{lem:CMT} 
Let $(Z^N(t))_{t\geq 0}$ be a binary branching process with individual birth rate $\lambda_N \geq 0$ and individual death rate $\mu_N \geq 0$. Further, let $\varphi_N$ and $b\in[0,1)$ as in Definition~\ref{def:moderate-selection} and denote $b_N=-\log_N\varphi_N$. Lastly, let $\beta_N\geq b_N$ for all $N$, assume that $\beta:=\lim_{N\to\infty}\beta_N$ exists and let the initial conditions satisfy $Z^N(0) = \lfloor N^{\beta_N} \rfloor$.
\begin{enumerate}[a)]
    \item \emph{(Slightly supercritical case.)}   
    Assume that $\lambda_N = 1+a \varphi_N$ for some $a>0$, and $\mu_N=1$.
   
    In case $\beta=b$ assume additionally that $N^{\beta_N}\gg \varphi_N^{-1}$. Then, for any $T>0$
    \[
      \lim_{N\to\infty} \P \Big(  \sup_{t \leq T} \Big| \log^+_N(Z^N(t \varphi_N^{-1} \log N) ) - (\beta+at) \Big| > \frac{1}{\log N} + |\beta_N-\beta| \normalcolor \Big)
        = 0;
    \]
    a fortiori,  $(\log^+_N(Z^N(t \varphi_N^{-1} \log N) ))_{t\in[0,T]}$ converges uniformly in probability to $(\beta+at)_{t\in[0,T]}$.
    \item \emph{(Slightly subcritical case.)}  Assume that $\mu_N = 1+a \varphi_N$ for some $a>0$, and $\lambda_N=1$.

    Let $\beta>b$ and $(\eps_N)_{N\in\N}$ be a non-negative sequence tending to 0 such that $N^{-\varepsilon_N}\to0$ as $N\to\infty$.
    Then, for  $T_N := \frac{\beta_N-b_N-\eps_N}{a}$,
    \[ \lim_{N\to\infty}  \P \Big(  \sup_{t \leq T_N} \Big| \log^+_N(Z^N(t \varphi_N^{-1} \log N) ) - (\beta-at) \Big| > \frac{1}{\log N} + |\beta_N-\beta| \normalcolor \Big)=0. \]
    In particular, for any $s<\frac{\beta-b}{a}$, $(\log^+_N(Z^N(t\varphi_N^{-1}\log N)))_{t\in [0,s]}$ converges uniformly in probability to $(\beta-at)_{t\in [0,s]}$ as $N\to\infty$.
\end{enumerate}
\end{lemma}
\normalcolor
\begin{proof}
  In both cases a) and b), put $r_N = \lambda_N-\mu_N$. Consider the martingale $(\widehat M_t^N)=(e^{-r_Nt}Z^N(t)-Z^N(0))_t$ and note that by \cite[Chapter V, p.103, Theorem~6.1]{Harris63},
  \begin{align}\label{classvar}
     \E[(\widehat M_t^N)^2]
     &= \mathrm{Var}[\widehat M_t^N] + 0
      = e^{-2r_Nt}\mathrm{Var}[Z^N(t)]
      = e^{-2r_Nt}Z^N(0)\cdot\frac{\lambda_N+\mu_N}{r_N}(e^{2r_Nt}-e^{r_Nt}).
  \end{align}
  Hence, by Doob's inequality, for any $\eta=\eta_N>0$
  \begin{align*}
    \P \Big( & \sup_{t \leq T \varphi_N^{-1} \log N} | e^{-r_Nt} Z^N(t) - Z^N(0)| \geq \tfrac12N^\eta \Big)
       = \P \Big( \sup_{t \leq T \varphi_N^{-1} \log N} |\widehat M_t^N | \geq  \tfrac12N^\eta \Big) \\
      &\leq 4 N^{-2\eta} \E[(\widehat M^N_{T \varphi_N^{-1} \log N})^2]
       =  4 Z^N(0)N^{-2\eta} \frac{\lambda_N+\mu_N}{r_N} (1-e^{-Tr_N\varphi_N^{-1}\log N}).
       \numberthis\label{eq:lastestimate}
  \end{align*}
  Now, in case a), this is asymptotically equivalent to 
  \[
    \tfrac4a N^{\beta_N+b_N-2\eta}(2+a\varphi_N)(1-N^{-aT})
      \in O(N^{\beta_N+b_N-2\eta}).
  \]
  Combined with the assumption that $N^{\beta_N}\gg N^{b_N}$, this implies that
  the probability of the event
  \[
    \Omega_1^N
     = \Big\{ \sup_{t \leq T \varphi_N^{-1} \log N} \big| e^{-r_Nt} Z^N(t) - N^{\beta_N}|
               \leq \tfrac12N^{\frac{2\beta_N+b_N}{3}} \Big\}
  \]
  tends to 1 as $N\to\infty$. On this event, division by $N^{\beta_N}$ leads to
  \[
    \sup_{t\leq T}\Big|\frac{Z^N(t\varphi_N^{-1}\log N)}{N^{\beta_N}e^{r_Nt\varphi_N^{-1}\log N}}-1\Big|
      \leq \tfrac12.
  \]
  Since $|\log x|\leq 2|x-1|$ for all $x\geq\frac12$ it follows
  \begin{align*}
      1
       &\geq \sup_{t\leq T}\Big|\log^+(Z^N(t\varphi_N^{-1}\log N))-\log(N^{\beta_N}e^{r_Nt\varphi_N^{-1}\log N})\Big|
        \\[.5em]
       &= \sup_{t\leq T}\Big|\log^+(Z^N(t\varphi_N^{-1}\log N))-(\beta_N+at)\log N\Big|
  \end{align*}
  and thus
  \begin{align*}
      \sup_{t\leq T}\Big|\log^+_N(Z^N(t\varphi_N^{-1}\log N))-(\beta+at)\Big|
        &\leq \frac{1}{\log N} + |\beta-\beta_N|
         \xrightarrow{N\to\infty} 0.
  \end{align*}
  %---------------------------------------------
  In case b), the last term in \eqref{eq:lastestimate} becomes equivalent to
  \[
   \tfrac4a  N^{\beta_N+b_N-2\eta}(2+a\varphi_N)(N^{aT_N}-1)
      \in O(N^{\beta_N+b_N-2\eta+aT_N}).
  \]
  We assume throughout the remainder of the proof that $N$ is sufficiently large such that $T_N>0$. Now, the probability of the event
  \[ 
    \Omega_2^N
      = \Big\{ \sup_{t \leq T_N \varphi_N^{-1} \log N} \big| e^{-r_Nt} Z^N(t) - N^{\beta_N} \big|
                \leq \tfrac12N^{\beta_N-\frac{\eps_N}{3}}  \Big\},
  \]
  tends to 1 as $N\to\infty$ since we assumed $N^{-\varepsilon_N}\to0$ and
  \[
    \beta_N + b_N - 2(\beta_N-\frac{\eps_N}{3}) + a T_N
      = -\frac{\eps_N}{3}.
  \]
  Analogously to a) it follows on $\Omega_2^N$ that
  \begin{align*}
      \sup_{t\leq T_N}\Big|\log^+_N(Z^N(t\varphi_N^{-1}\log N))-(\beta-at)\Big|
       &\leq \frac1{\log N}\sup_{t\leq T_N}\Big|\log^+(Z^N(t\varphi_N^{-1}\log N))-(\beta_N-at)\log N \Big|
               \\[.5em] & \qquad + |\beta-\beta_N| \\[.5em]
       &\leq \frac1{\log N} + |\beta-\beta_N|
        \xrightarrow{N\to\infty} 0.  \numberthis\label{goodFelix}
  \end{align*}
\end{proof}

\begin{proof}[Proof of Proposition~\ref{Zph2}] With $\beta_N:= b_N$ and $\beta:= b$, the assertion in Proposition~\ref{Zph2} is readily implied by Lemma~\ref{lem:CMT} a.)
\end{proof}
\begin{proof}[Proof of Proposition~\ref{Zph4}]
Let $\eps_N = \frac23 \frac{\log \log N}{\log N}$. Note that $N^{-\eps_N} \to 0$ and therefore $\eps_N$ satisfies the condition of Lemma~\ref{lem:CMT} b). Putting $t_N = \frac{1-b_N-\frac14 \eps_N}{a}$, \[ \beta_N = Y_0^N =\log_N \Big( N-\big\lfloor N(1-\frac{1}{\sqrt{\log N}})\big\rfloor \Big), \]
and
\[ A_N:=\Big\{ \sup_{0\le t_N}|Y_0^N(t) -(1-at) | \ge \frac{1}{\log N} + |\beta_N-1| \Big\}, \]
$\P(A_N)$ tends to $0$ as $N\to\infty$ thanks to Lemma~\ref{lem:CMT} b) applied with $\beta=1$. 

Now, note that the function $t \mapsto 1-at$ has the value $b_N + \frac14 \eps_N$ at time $t_N$, whereas $|\beta_N-1| \sim \frac12 \frac{\log\log N}{\log N}$ as $N\to\infty$. Thus, on the event $A_N^c$, we have
\[
  Y_0^N(t_N\varphi_N^{-1}\log N)
    \leq b_N+\tfrac14 \eps_N + \frac{1}{\log N} + |\beta_N-1|
    \leq b_N + \big( \tfrac16 + \tfrac12 + o(1) \big) \frac{\log \log N}{\log N}
    \leq b_N + \frac{\log\log N}{\log N}
\]
 when $N$ is sufficiently large. On the other hand, on $A_N^c$ it also holds that
\[
  \begin{aligned}
    Y_0^N\Big( \frac{1-b_N-3\eps_N}{a} \varphi_N^{-1}\log N \Big)
      & \geq b_N + 3\eps_N - \frac{1}{\log N}- |\beta_N-1| \\
      & = b_N + \Big( 2 - \frac12 - o(1) \Big) \frac{\log\log N}{\log N} \geq b_N + \frac{\log\log N}{\log N} 
  \end{aligned}
\]
for all $N$ sufficiently large.
Therefore,
\[ \frac{1-b_N-3\eps_N}{a} \leq \vartheta_N \leq \frac{1-b_N-\frac14 \eps_N}{a} \]
with probability tending to $1$. This implies~\eqref{convtheta1} as well as the uniform convergence in probability claimed in Proposition~\ref{Zph4}, and the proof is finished.
\end{proof}
\subsubsection{Proof of Proposition~\ref{Zph3}}
    With $r_N := a\varphi_N$,  $t_N:= \tfrac{(\log N)^{1/3}}{r_N}$ and $Z_1^N(0):=  \big\lfloor \frac N{\sqrt{\log N}} \big\rfloor $ it suffices to prove that 
\begin{equation}\label{Z1large}
    \P(Z_1^N(t_N) \ge N) \to 1 \quad \mbox{ as } N\to \infty.
\end{equation}
To this end we observe first that
\begin{equation}\label{expest0}
  \E[Z_1^N(t_N)]
    \sim \frac N{\sqrt{\log N}} e^{r_Nt_N}
    \sim \frac N{\sqrt{\log N}} e^{(\log N)^{1/3}}
    \gg N.
\end{equation}
and hence
\begin{equation}\label{expest}
  \E[Z_1^N(t_N)]-N\ge \frac N{2\sqrt{\log N}} e^{r_Nt_N} \quad \mbox{ for } N \mbox{ large enough}.
\end{equation}
From the variance formula~\eqref{classvar} we conclude that, for $N$ large enough,
\begin{equation}\label{varest}
  \sigma_N^2
   := {\rm Var}[Z_1^N(t_N)]
   \sim \frac N{\sqrt{\log N}}\frac{2+r_N}{r_N} \big(e^{2r_Nt_N}-e^{r_Nt_N}\big)
   \le \frac N{\sqrt{\log N}}\frac{3}{r_N}e^{2r_Nt_N}.   
\end{equation}
Making use of~\eqref{lowerboundphi} we conclude from~\eqref{expest} and \eqref{varest} that
\begin{equation}\label{kNlarge}
  d_N := \frac{\E[Z_1^N(t_N)]-N}{\sigma_N}\to \infty \quad{\text{as }} N\to \infty.  
\end{equation}
By construction,
\begin{equation}\label{firstbound}
    \P(Z_1^N(t_N) < N)
      \le \P\big( \big|Z_1^N(t_N) - \E[Z_1^N(t_N)]\big| \ge d_N\sigma_N\big).
\end{equation}
Finally, by Chebyshev's inequality, the r.h.s.\ of~\eqref{firstbound} is bounded above by~$\tfrac 1{d_N^2}$.
This concludes the proof of~\eqref{Z1large}, and hence that of Proposition~\ref{Zph3}. \qed

\subsection{Completion of the proof of Theorem~\ref{thm:house} part  I}\label{pfpart1old}
1. Let the Galton--Watson processes $Z^N_0$ and~$Z_1^N$ be as in Section~\ref{rescrb}. Consider first the situation of Proposition~\ref{Zph2},  assuming that $Z_1^N(0)=\big\lfloor \tfrac {\log N}{\varphi_N} \big\rfloor$. We claim that \begin{equation}\label{duration1}
    D_N:= \frac{\varphi_N}{\log N}T_{\big\lfloor N\big(1-\tfrac 1{\sqrt{\log N}}\big)\big\rfloor}^N \to \frac{1-b}a\quad \mbox{ in probability} \quad \mbox{as } N\to \infty.
\end{equation} 
To see this, observe for $t_0 \in \big(0, \tfrac {1-b}a\big)$ and $t_1 > \tfrac {1-b}a$ the relations
$$
  \P(D_N \le t_0)
    = \P\Big(\sup_{u\le t_0}Z_1^N\big(\tfrac{u\log N}{\varphi_N}\big)
    \ge N\big(1-\tfrac 1{\sqrt{\log N}}\big)\Big)
$$
and
$$
  \P(D_N > t_1)
    \le \P\Big(Z_1^N\big(\tfrac{t_1\log N}{\varphi_N}\big)
    \le N\big(1-\tfrac 1{\sqrt{\log N}}\big)\Big).
$$
Because of Proposition~\ref{Zph2} both of the right-hand sides converge to $0$ as $N\to \infty$, which shows~\eqref{duration1}.

 Combining this with Proposition~\ref{Zph1}  we see that \eqref{duration1} remains valid if we assume that $Z_1^N(0) =1$ and condition on the events $\{T_N^N < T_0^N\}$ (noting that the probabilites of these events are asymptotically  equivalent to those of the events 
 $\{Z_1^N \to \infty\}$ that appear in Proposition~\ref{Zph1}).   
 Again invoking Proposition~\ref{Zph2} we see that, for   $Z_1^N(0) =1$ and conditioned under the event $\{T_N^N < T_0^N\}$,  
for all $\eps > 0$
\begin{equation}\label{Zphase123}
 \Big(\log_N^+\Big(Z_1^N\Big(\big(t\wedge T^N_{\big\lfloor N\big(1-\tfrac 1{\sqrt{\log N}}\big)\big\rfloor}\big)\,\varphi_N^{-1}\log N\Big)\Big)\Big)_{t\ge \eps} \to \big(h_1(t)\big)_{t\ge \eps}^{}  
\end{equation}
uniformly in probability as $N\to \infty$. 

2. Now assume that $Z^N_0$ starts in $N-\big\lfloor N\big(1-\tfrac 1{\sqrt{\log N}}\big)\big\rfloor$ and is independent of $Z_1^N$. We define
\begin{equation}\label{Z0tilde}
   \widetilde Z^N_0(t)
    := 
     \begin{cases}N-Z_1^N(t)
       \quad &\mbox{for } t < T^N_{\big\lfloor N\big(1-\tfrac 1{\sqrt{\log N}}\big)\big\rfloor}\\
       Z^N_0\Big(t- T^N_{\big\lfloor N\big(1-\tfrac 1{\sqrt{\log N}}\big)\big\rfloor}\Big)  \quad &\mbox{for } t \ge T^N_{\big\lfloor N\big(1-\tfrac 1{\sqrt{\log N}}\big)\big\rfloor}.
   \end{cases}
\end{equation}
Combining~\eqref{Zphase123} with Propositions~\ref{Zph4} and~\ref{Zph5} we infer that for all $\eps >0$ 
\begin{equation}\label{Zphase45}
 \big(\log_N^+\big(\widetilde Z_0^N(t\,\varphi_N^{-1}\log N)\big)\big)_{t\ge 0} \to h_0   
\end{equation}
uniformly in probability on the set $D_0^\eps$ defined in~\eqref{defsetsD}.

3. As a companion to~\eqref{Z0tilde} we define
\begin{equation}\label{Z1tilde}
   \widetilde Z_1^N(t):= \begin{cases}Z_1^N(t) \quad &\mbox{for } t < T^N_{\big\lfloor N\big(1-\tfrac 1{\sqrt{\log N}}\big)\big\rfloor}\\
   N-  \widetilde Z^N_0(t)  \quad &\mbox{for } t \ge T^N_{\big\lfloor N\big(1-\tfrac 1{\sqrt{\log N}}\big)\big\rfloor}.
   \end{cases}
\end{equation}

4.
Following the program announced in Remark~\ref{remarkprel}~a) and further explained in Section~\ref{pfstrat}, 
we define the time change factor
\begin{equation}\label{deftcf}
    f^N(t):= \begin{cases}
        1-\tfrac{Z_1^N(t)}N \quad \mbox{for}\quad  t < T^N_{\big\lfloor N\big(1-\tfrac 1{\sqrt{\log N}}\big)\big\rfloor}\\
         1-\tfrac{\widetilde Z^N_0(t)}N \quad \mbox{for}\quad  t\ge T^N_{\big\lfloor N\big(1-\tfrac 1{\sqrt{\log N}}\big)\big\rfloor}
    \end{cases}
\end{equation}
and the time change
\begin{equation}\label{sigma}
\sigma_t^N:= \int_0^t f^N(u) \, {\rm d}u,\quad  t\ge 0.
\end{equation}
It follows from the definition of the jump rates  (or by a generator calculation) that
\begin{equation}\label{timechange}
  \Big(\widetilde Z^N_0(\sigma_t^N), \widetilde Z_1^N(\sigma_t^N)\Big)_{t\ge 0} 
    \overset{(d)}{=} \big(X_0^N(t), X_1^N(t)\big)_{t\ge 0}.
\end{equation}
By construction we have 
\begin{equation}\label{tc0}
   1\ge f^N(t) \ge 1- \tfrac 1{{\sqrt{\log N}}} \quad \mbox{for}\quad  0\le t \le T^N_{\big\lfloor \tfrac {N}{{\sqrt{\log N}}} \big\rfloor} 
\end{equation}
and w.h.p.\ as $N \to \infty$
\begin{equation}\label{tc1}
    1 \ge f^N(t) \ge 1- \tfrac 1{(\log N)^{1/3}} \quad \mbox{for}\quad  t \ge  T^N_{\big\lfloor N\big(1-\tfrac 1{\sqrt{\log N}}\big)\big\rfloor}.
\end{equation}
(To see the latter, we observe that~\eqref{tc1} is valid on the event 
$$
  \mathscr E_N
    := \Big\{\widetilde Z^N_0 \mbox{ hits } 0 \mbox{ before  }
         \tfrac N{(\log N)^{1/3}}\Big\}
$$
and that $\P\big(\mathscr E_N \big| \widetilde Z^N(0) = \tfrac 
N{(\log N)^{1/2}}  \big. \big) \to 1$ as $N\to \infty$, where the latter property follows by observing $\tfrac 
N{(\log N)^{1/2}} \ll \tfrac 
N{(\log N)^{1/3}} $ and applying optional stopping to the supermartingale $\widetilde Z^N_0$, or simply comparing $\widetilde Z^N_0$ to a driftless random walk to which the classical ruin problem can be applied.)

Also by construction we have
\begin{equation}\label{tc2}
 1 \ge f^N(t) \ge \tfrac 1{\sqrt{\log N}} \quad \mbox{for}\quad  T^N_{\big\lfloor \tfrac {N}{{\sqrt{\log N}}} \big\rfloor} \le t \le T^N_{\big\lfloor N\big(1-\tfrac 1{\sqrt{\log N}}\big)\big\rfloor}.    
\end{equation}
Together with Proposition~\ref{Zph3} this shows that $T^N_{\big\lfloor N\big(1-\tfrac 1{\sqrt{\log N}}\big)\big\rfloor}-T^N_{\big\lfloor \tfrac {N}{{\sqrt{\log N}}} \big\rfloor} $ is of smaller order than $\frac{\log N}{\varphi_N}$ for the time-changed branching process, and hence by \eqref{timechange} also for $X_1^N$. Combining this with~\eqref{tc0},~\eqref{tc1} and the above steps 1.\ and 2.,  we see that the proof of  Theorem~\ref{thm:house}, part I.\ is complete.\qed

\section{Proof of Theorem~\ref{thm:house} part II}\label{secpfM1top}

\subsection{A metric for $M_1$ convergence}\label{metricM1}
In this section we briefly summarize, in adequate generality, the basics required for the framework and the proof of Theorem~\ref{thm:house} part II. 
With $I$ being a fixed interval in $\R$, let
\begin{equation}\label{defD}
   \mathcal D_I:=\{f:I\to \R \mid f \mbox{ is right continuous and has left limits}\} 
\end{equation} 
\begin{definition}\label{defextgraph} Let $f\in \mathcal D_I$.
  \begin{enumerate}[a)]
    \item 
      The  {\em extended graph} of $f$ is defined as
      $$
        \mathbb G(f)
         := \bigcup_{t\in I}\big(\{t\} \times [\min(f(t-),f(t)), \max(f(t-),f(t))]\big).
      $$
      Thus, in  words from~\cite{K23}, the extended graph  ``contains the graph together with the straight lines connecting the two ends of a jump discontinuity''.
    \item 
      A~{\em parametric representation} of (the extended graph of) $f$ is a mapping $\gamma=(  \xi,\rho)$  which maps $[0,1]$ onto $\mathbb G(f)$ (where $  \xi(u) \in \R$ denotes the first coordinate and $\rho(u) \in \R$ the second coordinate of $\gamma(u) \in \mathbb R^2$)  and for which $u \mapsto \gamma(u)$ is non-decreasing with respect to the total order $\preceq$ on $\mathbb G(f)$ defined by
      $$
        (t_1,x_1)\preceq (t_2,x_2) :\Longleftrightarrow t_1\le t_2
          \mbox{ and } \big(t_1=t_2 \Rightarrow |x_1 - f(t_1-)| \le |x_2 - f(t_1-)|\big).
      $$
  \end{enumerate}
  The set of all parametric representations of $\mathbb G(f)$ will be denoted by $\mathscr P(\mathbb G(f))$.
\end{definition}
\begin{remark}\label{existpar}
  \begin{enumerate}[a)]
  \item 
      For a piecewise constant function $f$ with finitely many jumps one obtains a parametric representation (e.g.) by traversing the (horizontal and vertical segments of the) extended graph at constant speed.  For a general $f\in \mathcal D_I$, the  existence of a parametric representation follows from an argument given in~\cite[Remark 12.3.3]{Whitt}. 
    \item
      For the notion of $M_1$ convergence (and for generating the $M_1$ topology) the following metric on $\mathcal D_I$ was proposed by  \cite{skorokhod1956limit}:
      \begin{equation}\label{defdM1}
        d_{M_1}(f,g):= \inf\{\sup\limits_{u\in [0,1]}| \phi(u)-\gamma(u)| : \varphi \in \mathscr P(\mathbb G(f)), \, \gamma \in \mathscr P(\mathbb G(g))\}.
      \end{equation} 
      where $\vert \cdot \vert$ refers to the Euclidean norm on $\R^2$.
   \end{enumerate}
\end{remark}
The following elementary lemma will be used for the proof of Proposition~\ref{propeastwest} below .
\begin{lemma}\label{lem:m1-sup}
    For $I= [\alpha, \beta]$, $t\in I$ and $f\in \mathcal D_I$ let 
    $$
      \overline f(t)
        := \max_{\alpha \le t'\le t} f(t'), \qquad
      \widetilde f(t)
        := \max_{\beta +\alpha  -t \le r\le \beta} f(r)
    $$
    be the running maximum of $f$ and of its time-reversal, respectively.
    Then
    $${\rm a)} \quad d_{M_1}(f,\overline f)\le \sup_{t\in I} | \overline f(t)-f(t) |, \qquad  \quad {\rm b)}\quad  d_{M_1}(f,\widetilde f)\le \sup_{t\in I} | \widetilde f(t)-f(\beta + \alpha - t) |. $$
\end{lemma}
\begin{proof}
    Let $\gamma= (  \xi, \rho)$ be a parametric representation of $f$, and construct from it a parametric representation $\overline \gamma = (\overline   \xi, \overline \rho)$ of $\overline f$ by taking $\overline   \xi :=   \xi$ and $\overline \rho(u):= \sup_{0\le w\le u} \rho(w)$, $u\in [0,1]$. The first assertion then follows by the definition of $d_{M_1}$. The proof of the second assertion is analogous.
\end{proof}

\subsection{The accompanying Galton--Watson processes and their discrete-time embeddings}
\label{GWpart2}
As in Section~\ref{rescrb} we consider  Galton--Watson processes $Z_1^N$ and $Z_0^N$ with jump rates given in~\eqref{Zrates}.  Recalling the parameters $a$ and $b$ that enter via~\eqref{Zrates} and~\eqref{propphiN}, we abbreviate
$$t^{(a,b)}:= \tfrac {1-b}a,$$
and fix a number $\eps >0$. With
$Z_1^N(0) :=1$ and $Z_1^N$ conditioned under survival, define
 \begin{equation}\label{defY1}
    \left.
        \begin{array}{c}
                Y_1^N(t):= \log_N^+\big(Z^N_{1} (t \varphi_N^{-1} \log N)\big)\mathbf 1_{\{t\ge 0\}} \\[.5em]
                y_1(t):= (b+at) \mathbf 1_{\{t\ge 0\}}\phantom{AAAAAAAA} 
    	\end{array}
    \right\}\, ,\quad t \in [-\eps, t^{(a,b)}+1],
\end{equation}
and with $Z_0^N(0) :=N-\big\lfloor N\big(1-\tfrac 1{\sqrt{\log N}}\big)\big\rfloor$ define
\begin{equation}\label{defY0}
    \left.
        \begin{array}{c}
                Y_0^N(t):= \log_N^+\big(Z^N_{0} (t \varphi_N^{-1} \log N)\big) \\[.5em]
                y_0(t):= (1-at) \mathbf 1_{\{t\le t^{(a,b)}\}}\phantom{AAA} 
    	\end{array}
    \right\}\,, \quad t \in [0, t^{(a,b)}+\eps].
\end{equation}
Note that the starting value $Z_0^N(0)$ is chosen equal to the one in Proposition~\ref{Zph4}, and that the functions $y_1$ and $y_0$ are closely related to $h_1$ and $h_0$ defined in Section~\ref{mainresults}: $y_1$ is simply a restriction of $h_1$, while $y_0$ arises by first shifting and then restricting the function $h_0$.
 
The  next proposition, which will be key for the proof of Theorem~\ref{thm:house} part II, gives refinements  of Propositions~\ref{Zph1} and~\ref{Zph5}.
\begin{prop} \label{propeastwest} For $i \in \{0,1\}$,
$$d_{M_1}(Y_i^N, y_i) \to 0 \quad \mbox{in probability as } N\to \infty.$$   
\end{prop} 
Before turning to the proof, which will be  given in Section~\ref{pfsecGW1}, we briefly explain its strategy.
We know already from Proposition~\ref{Zph1} that the time needed for $Y^N_1$ to climb from height $0$ to height $b$ is asymptotically negligible as $N\to \infty$, so what matters now is the discrete-time embedding of $Y^N_1$. To get a handle on this we recall that the discrete-time embedding of $Z_i^N$ is a random walk with steps $\pm 1$ and probability $p = p(i,N)$ for making an upward step being 
\begin{eqnarray}
   p(i,N)= \begin{cases}
       \frac{1+a\varphi_N}{2+a\varphi_N} > \frac 12 \mbox{ for } i=1\\
       \frac1{2+a\varphi_N} < \frac 12 \mbox{ for } i=0.
   \end{cases} 
\end{eqnarray}
\begin{definition}\label{discreteW}
  Let  
  $$\mathbb W:=\{w=(w(0),w(1), \ldots)\mid w(n) \in \mathbb Z, w(n+1)-w(n) \in \{-1,1\}, n \in \N_0\},$$
  and for $p\in (0,1)$ let $(W, (\P_k^p)_{k\in \Z})$ be the canonical model of a (discrete-time) random walk with step size~1 and upward transition probability  $p$. (In particular, $\P_k^p(W(0)=k)=1$ , $\P_k^p(W(1)=k+1)=p$.)

  For $k \in \N_0$ and $w \in \mathbb W$ let
\begin{equation}\label{deftauk}
    \tau_k(w):= \min\{n\in \N_0\mid w(n) =k\}, \quad \tau_k^+(w):= \min\{n\in \N\mid w(n) =k\},
\end{equation}
be the first visit time of $k$ and the first return time to $k$, respectively, and let
\begin{equation}\label{defsigma}
    \sigma_k(w):= \max\{n\in \N_0\mid w(n) =k, \, w(j) \in \N \mbox{ for } j \in \{0, \ldots, n-1\}\}
\end{equation}
be the time of the last visit of $k$ before $w$ hits $0$.
\end{definition}

The next two lemmas will be key for the proof of Proposition~\ref{propeastwest}; their proof will be given (in slightly more generality) in the appendix.
\begin{lemma}\label{logfluc1} For $p> 1/2$ and $\ell \in \N_0$ write $\P_\ell^{p,+}$ for the probability measure $\P_\ell^p$ on $\mathbb W$ conditioned under the event $\{\tau_0^+(W) = \infty\}$. Then
    \begin{equation}\label{fluctpNnew}
   \max_{1\le k \le N} \max_{\tau_{k}(W)\le n < \tau_{k+1}(W)} \big(\log_N(k)- \log_N(W(n))\big) \to 0 \quad \mbox{ in } \P_1^{p(1,N),+}\mbox{-probability as } N\to \infty.
\end{equation} 
\end{lemma}
\begin{lemma}\label{logfluc0} For $w \in \mathbb W$, write
    $S(w):= \max\limits_{0\le n \le \tau_0(w)} w(n)$ for the maximum of the path $w$ before it first visits $0$, and abbreviate
    $k_N:= N-\lfloor N(1-\tfrac1{\sqrt{\log N}})\rfloor$.
    Then with $\sigma_k$ defined in \eqref{defsigma} 
    \begin{align}\label{flucdownfirst}
       \max_{0\le n <  \sigma_{S(W)}(W)}(\log_N (S(W)) - \log_N(W(n)))\to 0 
        &\quad \mbox{ in } \P_{k_N}^{p(0,N)}\mbox{-probability as } N\to \infty 
    \intertext{and}
    \label{flucdown2new}
         \max_{S(W)\ge k > 0}\,\max_{\sigma_{k}(W)\le n <  \sigma_{k-1}(W)}(\log_N (k) - \log_N(W(n)))\to 0
          &\quad \mbox{ in } \P_{k_N}^{p(0,N)}\mbox{-probability as } N\to \infty.
    \end{align}
\end{lemma}

\subsection{Proof of Proposition~\ref{propeastwest}}\label{pfsecGW1}
In this section we will prove the two parts (cases $i=1$ and $i=0$) of Proposition~\ref{propeastwest}, based on the fluctuation results on discrete-time embeddings that will be provided in Section~\ref{pfpropM1}.

\subsubsection{Case $i=1$}
Let $$\overline Y_1^N(t):= \max_{-\eps \le u \le t} Y_1^N(u),  \qquad  t \in [-\eps, t^{(a,b)}]$$
be the running maximum of $Y_1^N$ defined in~\eqref{defY1}. Since the discrete-time embedding of $Z_1^N$ (conditioned to survival) is  the Markov chain $W$ under the measure $\P_1^{p(1,N),+}$, we readily deduce from Lemma~\ref{logfluc1} the following
\begin{lemma} \label{distYYbar}
\begin{equation}\label{Ybar}
\sup_{-\eps \le t \le t^{(a,b)}+1}|\overline Y_1^N(t)- Y_1^N(t)| \to 0 \quad \mbox{in probability as } N\to \infty.
\end{equation}
\end{lemma}
An immediate consequence of Lemmas~\ref{distYYbar} and~\ref{lem:m1-sup} a) is the following statement.
\begin{cor}\label{cor:yybar-m1}
We have
  $$
    d_{M_1}(Y_1^N, \overline Y_1^N) \to 0 
      \quad \mbox{in probability as } N\to \infty.
  $$
\end{cor}
\begin{definition}
  We put 
  \begin{equation}\label{defzeta}
    \zeta_N
     := \inf\big\{t\ge 0 \mid Y_1^N(t) \ge \log_N\lfloor\tfrac{\log N}{\varphi_N}\rfloor\big\},
  \end{equation}
  and
  $$
    y_1^N(t) := 
    \begin{cases}
      0 &\mbox{for }-\eps\le t<0,\\
      \overline Y_1^N(t) & \mbox{for } 0 \le t \le \zeta_N,\\
      \log_N^{} \lfloor \tfrac {\log N}{\varphi_N}\rfloor  +a(t-\zeta_N)
        &\mbox{for } \zeta_N\le t\le t^{(a,b)}+1. 
    \end{cases}
  $$
\end{definition}
\begin{remark}\label{zetaremark}
  \begin{enumerate}[a)]
    \item From~\eqref{defzeta} and the definitions of $Y_1^N$ and $T_{\big\lfloor \tfrac{\log N}{\varphi_N}\big\rfloor}$ we have
    \begin{equation*}
        \zeta_N
        = \tfrac{\varphi_N}{\log N} T_{\big\lfloor \tfrac{\log N}{\varphi_N} \big\rfloor}.
    \end{equation*}
    \noindent Hence Proposition~\ref{Zph1} implies
    \begin{equation}\label{zetasmall}
      \zeta_N \to 0 \quad \mbox{in probability as } N\to \infty.
    \end{equation}
    \item We now will argue that  
    \begin{equation}\label{hNclose}
      d_{M_1} (y_1^N, y_1)
        \to 0 \quad \mbox{in probability as }  N\to \infty.
    \end{equation}
  \end{enumerate}
Indeed, the extended graph of $y_1^N \big |_{[-\eps,\zeta_N]}$ consists of vertical and horizontal segments and has length $\log_N^{} \lfloor \tfrac {\log N}{\varphi_N}\rfloor+\eps+\zeta_N$. Let $\Gamma_N=(\Xi_N, R_N)$ be the parametric representation of $y_1^N$ which traverses its extended graph at constant speed, and let $u_N$ be the parameter value in $[0,1]$ for which $R_N$ reaches height $\log_N^{} \lfloor \tfrac {\log N}{\varphi_N}\rfloor$. Now consider that parametric representation $\widetilde \Gamma_N$ of $y_1$ whose height component $\widetilde R_N(u)$ is identical to $R_N(u)$ for $u \le u_N$, and which for $u\ge u_N$ moves along the extended graph of $y_1$ at constant speed. After the parameter time $u_N$ both height components are continuous, and $y_1$ and $y_1^N$ have the same slope above height $\max(b, \log_N\lfloor \tfrac {\log N}{\varphi_N}\rfloor)$. Hence, it follows readily that
\begin{align*}
  \sup_{u\in[0,1]}|\Gamma_N(u)-\widetilde \Gamma_N(u)|
   & = \sup_{u\in[0,u_N]}|\Xi_N(u)-\widetilde\Xi_N(u)| + \sup_{u\in(u_N,1]}|R_N(u)-\widetilde R_N(u)| \\
   & \leq \zeta_N + |\log_N\lfloor\tfrac{\log N}{\varphi_N}\rfloor - b|,
\end{align*}
which together with~\eqref{zetasmall} implies~\eqref{hNclose}. 
\end{remark}

\begin{lemma}\label{lem:Ybar-h1N}  As $N\to \infty$,
  $$
    \sup_{-\eps \le t \le t^{(a,b)}+1}|\overline Y_1^N(t)- y_1^N(t)|
      \to 0 \quad \mbox{in probability}.
  $$
\end{lemma}
\begin{proof}  For $ Z_1^N$ being  a Galton--Watson process with jump rates~\eqref{Zrates} and  $Z_1^N(0) = \lfloor \tfrac{\log N}{\varphi_N} \rfloor$  we have by Proposition~\ref{Zph2}

\begin{equation}
\sup_{0\le t \le t^{(a,b)}+1} \big |\log_N Z_1^N(t\varphi_N^{-1}\log N) -(b+at)\big | \to 0 \mbox{ in probability as } N\to \infty.
\end{equation}
By the strong Markov property 
this implies
\begin{equation}
\sup_{\zeta_N\le t \le t^{(a,b)}+1} \big |Y_1^N(t) -(b+a(t-\zeta_N))\big | \to 0 \mbox{ in probability as } N\to \infty.
\end{equation}
By Lemma~\ref{distYYbar} combined with~\eqref{propphiN} we thus also have
\begin{equation}\label{westroof}
\sup_{\zeta_N\le t \le t^{(a,b)}+1} \big |\overline Y_1^N(t) -(\log_N^{} \lfloor \tfrac {\log N}{\varphi_N}\rfloor +a(t-\zeta_N))\big | \to 0 \mbox{ in probability as } N\to \infty.
\end{equation}
Since $\overline Y_1^N(t)$ and $y_1^N(t)$ are equal for $t\in [-\eps, \zeta_N]$, the assertion of the lemma is proved.
\end{proof} 
Because both $\overline Y_1^N(t)$ and $y_1^N(t)$ are non-decreasing in $t$, it is readily checked (by an argument similar to the proof of Lemma~\ref{lem:m1-sup}) that $d_{M_1}(\overline Y_1^N, y_1^N) \le \sup_{-\eps \le t \le t^{(a,b)}+1}|\overline Y_1^N(t)- y_1^N(t)|$. Hence from Lemma~\ref{lem:Ybar-h1N} we obtain
\begin{cor}\label{cor:Ybar-h1N} As $N\to \infty$,
  $$
    d_{M_1} (\overline Y_1^N, y_1^N )
      \to 0 \mbox{ in probability}.
  $$ 
\end{cor}
Proposition~\ref{propeastwest} for $i=1$ is now  an immediate consequence of the triangle inequality, combining~\eqref{hNclose} with Corollaries~\ref{cor:yybar-m1} and~\ref{cor:Ybar-h1N} .
\subsubsection{Case $i=0$}
Let $Y_0^N$ be as in~\eqref{defY0}.
We set
$$\widetilde Y_0^N(t):= \max_{0\le v \le t^{(a,b)}+1-t} Y_0^N(t^{(a,b)}+1-v),  \qquad  t \in [0, t^{(a,b)}+1].
$$
As an immediate consequence of Lemma~\ref{logfluc0} we obtain
\begin{lemma} \label{distYY0bar}
\begin{equation}\label{Y0bar}
\sup_{0 \le t \le t^{(a,b)}+1}|\widetilde Y_0^N(t)- Y_0^N(t)| \to 0 \quad \mbox{in probability as } N\to \infty.
\end{equation}
\end{lemma}
Combining Lemma~\ref{distYY0bar} with an application of Lemma~\ref{lem:m1-sup} b)  (with $\alpha:= 0$ snd $\beta:= t^{(a,b)}+1$) we readily infer that
\begin{equation}\label{yy0bar}
    d_{M_1}(Y_0^N, \widetilde Y_0^N) \to 0 
      \quad \mbox{in probability as } N\to \infty.
\end{equation}
Recalling that $Y_0^N$ is assumed to start in $\log_N\big(N-\big\lfloor N\big(1-\tfrac 1{\sqrt{\log N}}\big)\big\rfloor\big)\sim 1-\frac 12 \log_N\log N$ we define
$$\vartheta^N_{4,5} := \inf\Big\{t \ge 0: Y_0^N \le b_N+\log_N \log N\}$$ 
and
$$\vartheta^N_{5,0} := \inf\{t \ge 0: Y_0^N = 0\}.$$
From Propositions~\ref{Zph4} and~\ref{Zph5} we conclude that the following three assertions hold as $N\to \infty$
\begin{equation}\label{newtheta}
\vartheta^N_{4,5} \to t^{(a,b)} \text{ in probability}, 
\end{equation}
\begin{equation}\label{supconv}
\mathbb P(\sup_{0\le t\le \vartheta^N_{4,5}}|Y_0^N(t) -(1-at) | \ge \log_ N\log N) \to 0,
\end{equation}
\begin{equation}\label{therest}\vartheta^N_{5,0}-\vartheta^N_{4,5} \to 0 \quad \mbox{ in probability}.
\end{equation}

\begin{definition}
  We put
  \normalcolor
  $$
    y_0^N(t) := 
    \begin{cases}
      Y_0^N(0)  -at &\mbox{for } 0\le t< \vartheta^N_{4,5},\\
      \widetilde Y_0^N(t) & \mbox{for } \vartheta^N_{4,5} \le t < t^{(a,b)}+1.
    \end{cases}
  $$
\end{definition}
\begin{remark}
  \begin{enumerate}[a)]
    \item From \eqref{supconv} we conclude that
    \begin{equation}\label{Zph4appl}
      \sup_{0\le t\le \vartheta^N_{4,5}} | Y_0^N(t)-y_0^N(t) | \to 0 \quad\mbox{in probability as } N\to \infty.
    \end{equation}
    Since $\widetilde Y_0^N(t)$ and $y_0^N(t)$ are equal for $t\in [\vartheta^N_{4,5}, t^{(a,b)}+1]$, we conclude from Lemma~\ref{distYY0bar} that
    \begin{equation}\label{yYtilde}
      \sup_{0 \le t \le t^{(a,b)}+1}|\widetilde  Y_0^N(t)- y_0^N(t)|
      \to 0 \quad \mbox{ in probability as } N\to \infty.  
    \end{equation}
    Both $\widetilde  Y_0^N(t)$ and $y_0^N(t)$ are non-increasing in $t$, hence an argument paralleling the proof of Lemma~\ref{lem:m1-sup} shows that the $d_{M_1}$-distance is upper-bounded by the sup-distance, thus yielding
    \begin{equation}\label{Y0y0N}
      d_{M_1}(\widetilde Y_0^N, y_0^N) \to 0 \mbox{ in probability as } N\to \infty.
    \end{equation}

    \item As in Remark~\ref{zetaremark} we infer, now applying~\eqref{therest} in place of Proposition~\ref{Zph1}, that
    \begin{equation}\label{y0y0N}
      d_{M_1}(y_0^N, y_0) \to 0 \mbox{ in probability as } N\to \infty.
    \end{equation}
  \end{enumerate}
\end{remark}
The assertion of Proposition~\ref{propeastwest} for $i=0$ now follows from~\eqref{yy0bar}, \eqref{Y0y0N}, \eqref{y0y0N} and the triangle inequality.
\subsection{Completion of the proof of Theorem~\ref{thm:house} part II}\label{secpfM1topold} Following the strategy outlined in Section~\ref{pfpart1}, let $Z_1^N$ and $Z_0^N$ be independent Galton--Watson processes with jump rates given in~\eqref{Zrates}. The initial state of $Z_1^N$ at time $0$ is $1$, and that of $Z_0^N$ is  $k_N= N-\big\lfloor N\big(1-\tfrac 1{\sqrt{\log N}}\big)\big\rfloor$, as defined in Lemma~\ref{logfluc0}.
Let $Z_1^N$ be conditioned to survival. With $T_k^N$ defined as in~\eqref{def:TNk} we abbreviate
$$T_{2,3}^N :=T^N_{\big\lfloor \tfrac {N}{\sqrt{\log N}}\big \rfloor}, \quad T_{3,4}^N := T^N_{\big\lfloor N\big(1-\tfrac 1{\sqrt{\log N}}\big)\big\rfloor}.$$
In the jargon of Section~\ref{pfpart1} these two stopping times stand for the times of transitions from phase 2 to phase 3, and from phase 3 to phase 4, respectively. 
 Define $(\widetilde Z_0^N, \widetilde Z_1^N)$ as in~\eqref{Z0tilde}, \eqref{Z1tilde}.

The pair $(X_1^N, X_0^N)$  is now constructed via the time-change specified in~\eqref{timechange} from the pair $(\widetilde Z_1^N, \widetilde Z_0^N)$. The  pair $(H_1^N, H_0^N)$ figuring in Theorem~\ref{thm:house} is obtained from $(X_1^N, X_0^N)$ via  the logarithmic scaling~\eqref{logscaling}, and  the pair $(Y_1^N, Y_0^N)$ is obtained by the same scaling from 
$(Z_1^N, Z_0^N)$, see~\eqref{logscal}. The $Z_1^N$-measurable stopping times $T_{2,3}^N$ and $T_{3,4}^N$  are translated into stopping times for $Y_1^N$, $X_1^N$ and $H_1^N$ as follows:
\begin{eqnarray*}
&\vartheta_{2,3}^N :=\tfrac{\varphi_N}{\log N} T_{2,3}^N, \qquad &\vartheta_{3,4}^N :=  \tfrac{\varphi_N}{\log N} T_{3,4}^N,\\
& \Theta_{2,3}^N := \inf\{t \ge 0: X_1^N(t) \ge \tfrac N{\sqrt{\log N}}\} , \quad &\Theta_{3,4}^N := \inf\{t \ge 0: X_1^N(t) \ge \lfloor N\big(1-\tfrac 1{\sqrt{\log N}}\big)\big\rfloor\}, \\
&\tau_{2,3}^N := \tfrac{\varphi_N}{\log N}  \Theta_{2,3}^N, \qquad &\tau_{3,4}^N :=  \tfrac{\varphi_N}{\log N} \Theta_{3,4}^N.
\end{eqnarray*}
Theorem~\ref{thm:house} part 1 implies
\begin{equation}\label{convtheta}
\vartheta_{2,3}^N \to \frac{1-b}a \mbox{  in probability as }N\to \infty
\end{equation}
From~\eqref{tc0} we conclude that the
 time change factors $f_N$ defined in~\eqref{deftcf} are in phases 1 and 2 uniformly close to $1$ as $N \to \infty$, 
Together with~\eqref{convtheta} this implies
\begin{equation}\label{convtauN1}
\tau_{2,3}^N \to \frac{1-b}a \mbox{  in probability as } N\to \infty.
\end{equation}
On the other hand the argument following~\eqref{tc2}) yields
\begin{equation}\label{convtauN2}
\tau_{3,4}^N -\tau_{2,3}^N  \to 0 \mbox{  in probability as } N\to \infty.
\end{equation}
Together with Proposition 4.4 (case $i=1$) this implies
$$d_{M_1}\big(H_1^N\mid_{[-\varepsilon, \tau_{3,4}]}, y_1\mid_{[-\varepsilon, \tau_{3,4}]}\big) \to 0 \quad \mbox{in probability as }N\to \infty,$$
whence
$$d_{M_1}\big(H_1^N\mid_{[-\varepsilon, 2\frac{1-b}a +\varepsilon]}, h_1\mid_{[-\varepsilon, 2\frac{1-b}a +\varepsilon]}\big) \to 0 \quad \mbox{in probability as }N\to \infty.$$
Turning to $H_0^N$, we observe that by definition
$$H_0^N(\tau_{3,4}^N) = \log_N(k_N) \to  1 \mbox{ as } N\to  \infty.$$
We define
$$\tau_{5,0}^N := \inf\{t>0: H_0^N(t) = 0\}.$$
From~\eqref{tc1} we conclude that the
 time change factors $f_N$ defined in~\eqref{deftcf} are in phases 4 and 5 with high probability uniformly close to $1$ as $N \to \infty$. Combined with~\eqref{newtheta} and\eqref{therest} this implies
 \begin{equation}\label{4and5}
 \tau^N_{5,0}-\tau^N_{3,4} \to  \frac{1-b}a \quad\mbox{ in probability as } N\to \infty.
 \end{equation}
 Together with Proposition 4.4 (case $i=0$) this implies
$$d_{M_1}\big(H_0^N\mid_{[-\varepsilon, 2\frac{1-b}a +\varepsilon]}, h_0\mid_{[-\varepsilon, 2\frac{1-b}a +\varepsilon]}\big) \to 0 \quad \mbox{in probability as }N\to \infty,$$
thus completing the proof of  Theorem~\ref{thm:house} part II.
\qed
\section{Discussion and Outlook}\label{sec:discu}
\subsection{The five phases of a selective sweep and the ``house theorem''}
The present article  studies a two-type Moran model with population size $N$, consisting of initially  \mbox{$N-1$} resident (wild type, type 0) individuals and one beneficial mutant (type 1), as would be the case immediately after a mutation from type $0$ to type $1$. The mutant has selective advantage $s_N= a\varphi_N$, where $a>0$ is fixed and $\varphi_N \to 0$ with $-\log_N(\varphi_N) \to b \in [0,1)$, which corresponds to moderate selection with exponent $b$ for $b\in (0,1)$, and quasi-strong selection for $b=0$. To keep the exposition simple, we will assume $\varphi_N = N^{-b}$ in this section.

Conditional on the mutant's offspring escaping the initial genetic drift (which leads to its eventual fixation), the mutant population undergoes a selective sweep: it first grows from a single individual to a macroscopic fraction of the population, after which the original resident population declines to extinction. A short discussion of the relevance of selective sweeps in connection with so-called genetic hitchhiking is given in Section~\ref{secIntro}; in this context we also refer to the papers~\cite{durrett2004approximating} and~\cite{schweinsbergdurrett05}.  

The frequency of the mutant during a selective sweep is known to follow a logistic curve. The goal of our paper is to zoom in on the exponential grow and decline phases of the selective sweep and to study how the logarithmic frequencies (that is the base $N$ logarithms of the sizes) of the mutant and original resident evolve. The relevant timescale to investigate this question is $N^b \log N$, termed the {\em Gerrish--Lenski timescale} (or GL timescale for short) because~\cite{GL98}  considered beneficial mutations arriving on this scale, in order to study the phenomenon of clonal interference.   

Over the GL timescale, the selective sweep is found to have five phases, and the limiting trajectory (log-frequency against time) follows the shape of a ``house''. The left part of the house represents the growth of the mutant, the right part the decline of the original resident.

\begin{enumerate}[1]
\item[$\bullet$] Phase 1, `left wall of the house': On a timescale faster that $ N^b \log N$, the number of mutant individuals increases from $0$ to slightly more than $N^b$, and hence on the GL timescale the log-frequency graph is close to a vertical line between $0$ and $b$ in this phase.
\item[$\bullet$] Phase 2, `left part of the roof':  After surpassing $N^b$, the number of mutants grows exponentially at rate $aN^{-b}$, until it reaches almost the order of $N$. The corresponding logarithmic frequency trajectory looks like a straight line of slope $a$.
\item[$\bullet$] Phase 3, `top of the roof': On a timescale faster than $N^b \log N$, the size of the mutant, which before was $\ll N$, now becomes of order $N$, and the size of the former resident, which before was of order $N$, now becomes $\ll N$. 
The change in the logarithmic frequency is not visible in the limit $N\to \infty$.
\item[$\bullet$] Phase 4, `right part of the roof': The former resident declines exponentially until it is down to approximately $N^b$. The corresponding logarithmic frequency trajectory looks like a straight line of slope $-a$.
\item[$\bullet$] Phase 5, `right wall of the house': The former resident decreases from size approximately $N^b$ to extinction. This happens in a time that is asymptotically negligible in the GL timescale, so the log-frequency graph is close to a vertical line between $b$ and $0$.
\end{enumerate}

The main result (Theorem~\ref{thm:house}) states convergence in probability of the appropriately rescaled fixation times, uniform convergence in probability of the logarithmic frequencies to the 'roof of the house' during Phases 2--4, and  convergence in probability of the logarithmic frequencies in the Skorokhod $M_1$ topology  over the whole selective sweep. 
\subsection{Clonal interference with moderate selection: outlook and conjecture}\label{CImoderate}A mathematical model for clonal interference in the regime of strong selection (i.e.\ for $\varphi_N =1$) and mutants arriving with random fitness increments on the GL timescale, together with a corresponding scaling limit result, has recently been established in~\cite{HGSTW24}, leading to a system of {\em Poissonian interacting trajectories} (PIT). The present paper is also a first step in extending this program to the regimes of quasi-strong and moderate selection. 

We now give a brief outlook to a potential extension of the results of~\cite{HGSTW24}, sharpening statements in Section~7.2 of that paper in the light of the present results.  For an illustration, see  the lower panels of Figure~\ref{fig:storyline}. Throughout this section, we assume (in addition to~\eqref{propphiN}) that $\varphi_N \to 0$ as $N\to\infty$.

For some $\lambda > 0$, assume that at each of the times $T$ of a rate $\frac \lambda{\log N}$-Poisson process $\Pi^N$  on $\mathbb R_+$, a  beneficial mutation appears in the population, affecting a randomly chosen individual  among all the $N$ individuals alive at time $T$, and adding $A(T)\varphi_N$ to its fitness level, where the $A(T)$ are i.i.d.~$\mathbb R_+$-valued random variables with distribution $\gamma$ having a finite first moment.  Denoting by $X^N(M,t)$ the number of individuals alive at time $t$ that have fitness level $M$, let the replacement of a (randomly chosen) individual of fitness level $M'$ by a (randomly chosen) individual of fitness level~$M$ happen at rate $\frac 1N(1+(M-M')^+)X^N(M,t)X^N(M',t)$, i.e.\ neutral resampling per pair of individuals happens at rate $1/N$, while a higher fitness level conveys a linear bonus to the reproduction rate. Initially, we assume a homogenous population, with all individuals having fitness level 0.

Recalling the assumption~\eqref{propphiN}, we define $\widetilde \Pi^N:= (T^{1,N},T^{2,N}, \ldots)$ as the point process consisting of all those $T\in \Pi^N$ for which the offspring of the mutant at time $T$ ever reaches the size $N^b \log N$ (and thus, having overcome the fluctuations caused by genetic drift, becomes a {\em contender} for becoming resident).
With $X_0^N(t)$ being the size at time $t$ of the original resident, and (for $i=1,2,\ldots$) $X_i^N(t)$ being the size at time $t$ of the offspring of the mutant that appeared at time $T^{i,N}$,  define
$$
  H_i^N(t)
   := \log_N^+X_i^N\Big(\frac {t\log N}{\varphi_N}\Big),
   \quad i=0,1,2,\ldots
$$
Let  $(T_i)_{i=1,2,\ldots}$ be the points of a rate $\lambda$ Poisson point process on $\mathbb R_+$, and let $(\widehat A_i)_{i=1,2,\ldots}$ be an i.i.d.\ sequence of random variables whose distribution is the size-biasing of $\gamma$. (Note that, due to the assumption $\varphi_N \to 0$ as $N\to \infty$, the size-biasing of $\gamma$ is asymptotically equivalent to the biasing by the survival probability.)

 For $b\in [0,1)$ and for given $(T_i, \widehat A_i)_{i=1,2,\ldots}$ let $(H_i)_{i=0, 1,2,\ldots}$ be a system of piecewise linear $[0,1]$-valued trajectories which evolves deterministically as follows  (cf.\ Figure~\ref{fig:storyline},  bottom right):

\begin{enumerate}[1]
\item[$\bullet$] $H_0$ starts at time $0$ at height 1 with slope 0.
\item[$\bullet$] For $i=1,2,\ldots$, $H_i(t) = 0$ for $t\in [0,T_i)$,  $H_i(T_i)=b$  and $H_i$ starts at time $T_i$ with slope~$\widehat A_i$.
\item[$\bullet$] Every trajectory that reaches height 1 with slope $v$ at some time $t$,  instantly changes its slope to~$0$ and transmits this kink to all trajectories that have a height $> b$ at time $t$, i.e.\ instantly lowers their slopes by $v$.
\item[$\bullet$] As soon as a trajectory reaches height $b$ from above, it jumps to height $0$, where it then remains forever.
\end{enumerate}

\begin{conj} Assume $b\in(0,1)$, or $b=0$ and $\varphi_N \to 0$. Then, as $N\to \infty$, 
$$(H_i^N)_{i=0, 1,2,\ldots} \overset d \longrightarrow (H_i)_{i=0, 1,2,\ldots} \quad \mbox{ as random elements of the product space } \mathcal D(\mathbb R_+)^\mathbb N,$$
where the space $ \mathcal D$ is equipped with the Skorokhod $M_1$~topology.
\end{conj}

\subsection{Two near-critical Galton--Watson processes accompanying a selective sweep}

\mbox{A major} tool in the proof is a state-dependent time-change which allows to construct the sizes of the mutant and of the original resident from near-critical Galton--Watson processes.

Indeed the entire selective sweep curve in all five phases (as well as the ``limiting house'') can be obtained from either 

\begin{enumerate}[1]
\item[$\bullet$] a sequence $(\widehat Z_1^N)$ of slightly supercritical Galton--Watson processes with $\widehat Z_1^N$ started in $1$, conditioned to reach $N$ before $0$ and with holding times in $k$ multiplied by   $1-\frac kN$, 
\end{enumerate}
\noindent
or (alternatively) from

\begin{enumerate}[1]
\item[$\bullet$] a sequence $(\widehat Z_0^N)$ of slightly subcritical Galton--Watson processes with $\widehat Z_0^N$ started in $N-1$, conditioned to reach $0$ before $N$ and (again) with holding times in $k$ multiplied by   $1-\frac kN$. 
\end{enumerate}

It is readily checked that then both $(N- \widehat Z_1^N, \widehat Z_1^N)$ and  $(\widehat Z_0^N, N-\widehat Z_0^N)$ have the jump rates~\eqref{Morandyn}, hence according to Theorem~\ref{thm:house} their logarithmic frequencies on the GL timescale both converge to the house with \mbox{wall-height} $b$ and \mbox{roof-slope} $a$. In phase 1 we see a quick fight (and victory) against the fluctuations, leading to an escape from $0$ for $\widehat Z_1^N$, and to an escape from $N$ for $\widehat Z_0^N$. This was discussed in our Introduction, in the footsteps of  heuristic explanations from~\cite{barton1998effect} and~\cite{baake2019modelling}.  
In phase 5 we see a quick approach to the boundary, which then is $N$ for $\widehat Z_1^N$ and $0$ for $\widehat Z_0^N$. Here the conditioning (as well as the Malthusian growth) becomes irrelevant, and it is solely the (neutral) fluctuations that dictate that the duration of this phase is negligible on the GL timescale. For this as well a heuristics was given in the Introduction. The corresponding statements for the onset phase of the slightly supercritical Galton--Watson (GW for short) processes and the ``offset phase''  of the slightly subcritical GW processes, leading to an upward resp.\ a downward jump in the logarithmic scaling, are contained in Propositions~\ref{Zph1} and \ref{Zph5}.  Propositions~\ref{Zph2} and \ref{Zph4} (on the Malthusian growth and decline phase of slightly super- resp.\ subcritical GW processes) build on Lemma~\ref{lem:CMT}, which provides a sharpening of \cite[Lemma A.1]{CMT21}. We feel that these results might also be a contribution of interest to the theory of nearly critical Galton--Watson processes that was initiated by~\cite{kolmogorov38} (English translation: \cite[Chapter 25]{kolmogorov92}). 

In our proof of Theorem~\ref{thm:house} we swap the driving horses in the middle of the sweep, by using $\widehat Z_1^N$ as the source of randomness in phases 1, 2 and 3, and $\widehat Z_0^N$ (with the appropriate initial condition) in phases 4 and 5. This is convenient because then the time change factors appearing in phases 1 and 2 between $\widehat Z_1^N$ and the size of the mutant, and in phases 4 and~5 between $\widehat Z_0^N$ and the size of the resident, all converge to 1 as $N\to \infty$.

Theorem~\ref{thm:house} zooms into phase 1 even closer. While its first part  (based on  Proposition~\ref{Zph1}) tells us that the logarithmic frequency of the mutant climbs from $0$ to $b$ in asymptotically negligible GL time, the second part of Theorem~\ref{thm:house} (formulated in terms of Skorokhod $M_1$ convergence) implies that this happens in a nearly monotonic way as $N\to \infty$. This result is based on the discrete-time embedding
of $\widehat Z_1^N$, which in phase 1 can be compared to a driftless random walk started in 1 and conditioned to escape the origin. This allows to apply a result of Hambly, Kersting and Kyprianou on the law of the iterated logarithm of such random walks, see Proposition~\ref{HKKcor} in the appendix. Phase 5 is treated in a similar way by applying time reversal arguments.

\subsection{Quasi-weak selection: zooming in on the roof} The regime
$$
  \varphi_N \gg \tfrac 1N
  \quad\text{and}\quad
  b_N:=-\log_N\varphi_N\to1
$$ can be seen as that of {\em quasi-weak selection}. With $a > 0$ fixed and $N\in \mathbb N$,  let $X^N$ be as in Definition~\ref{def:Moran-Model} and $H^N$ as in Definition~\ref{defHN}.  Even though this regime is not captured by Theorem~\ref{thm:house}, in the light of that theorem one may conjecture that then $(H_0^N, H_1^N)$  converges as $N\to \infty$ in the Skorokhod~$M_1$ topology to the pair $(\mathbf 1_{(-\infty,0)}, \mathbf 1_{[0,\infty)})$ in probability. In this way, as $N \to \infty$ the roof of the house would collapse to a point destroying the picture of the house and hence the prospect of observing clonal interference.

By another scaling scheme however, we can zoom in on the roof part of the house. For this, denote by $h^b_i,i\in\{0,1\}$, the house functions given in \eqref{def-house}, this time parametrized in $b\in[0,1)$. Now, the functions $h^0_i$ represent a roof without walls, i.e.\ a tent, and we compute that for any $b,i,t$,
$$
  h^0_i(t)
    = \frac{\big(h^b_i(t\cdot(1-b))-b\big)^+}{1-b}
$$
(cf.\ Figure~\ref{fig:single_invader}, bottom right). From this observation we conjecture that, if $b_N$ converges to $1$ sufficiently slowly, 
$$
  \widetilde H_i^N(t)
   := \frac{\big(H_i^N(t\cdot(1-b_N))-b_N\big)^+}{1-b_N}, \quad i=0,1,
$$
converges in a suitable topology to the pair $(h^0_0, h^0_1)$.

\subsection{Tents and houses in the diffusion approximation of a selective sweep} 
Instead of individual-based models (like the Moran model used in this paper) one could  think of working right away with their (Wright--Fisher) diffusion approximations. These are usually run on the evolutionary timescale, with one time unit corresponding to $N$ generations. According to~\cite[formula (2.1)]{EPW06}, a Wright--Fisher diffusion $X$ with selection coefficient $\alpha = s_NN$, when started in $X(0)=0$ and conditioned (or, more precisely, $h$-transformed) to reach $1$, satisfies the stochastic differential equation
\begin{equation}\label{SDE1}
dX = \sqrt{ 2X(1-X)} \, dW +  s_N N X(1-X)\coth\Big(\frac {s_N^{}N} 2 X\Big)\,dt,   \quad t > 0,
\end{equation}
with $W$ denoting standard Brownian motion.

The Gerrish--Lenski timescale then corresponds to the scale~$\frac{t\log N}{s_N^{}N}$. Consequently, on that timescale the  SDE~\eqref{SDE1} becomes
\begin{equation}\label{SDE2}
d\xi^N = \sqrt{\tfrac {2\log N}{s_N^{} N} \xi^N(1-\xi^N)} \, dW +  \xi^N(1-\xi^N)\coth\Big(\frac {s_N^{}N} 2 \xi^N\Big) \log N\,dt,   \quad t > 0.
\end{equation}

For moderate selection $s_N:= aN^{-b}$ we conjecture that the sequence of $[0,1]^2$-valued diffusion processes $(1+\log_N(1-\xi^N), 1+ \log_N \xi^N)$ converges as $N\to \infty$  in the Skorokhod $M_1$~topology to the ``house'' (with wall height $b$ and roof slope $a$)  appearing in Theorem~\ref{thm:house}. In a sense this would close a circle since the law of the iterated logarithm for  Bessel-like random walks, which we used in our proof of Theorem~\ref{thm:house} (see Proposition~\ref{HKKcor} and formula~\eqref{limsup} below), was derived  in~\cite{hambly2003law} via Skorokhod embedding from a corresponding law for Bessel diffusions.

\appendix

\section{Logarithmic fluctuations of Bessel-like random walks}\label{pfpropM1}
The main objective of this section is the proof of Lemmas~\ref{logfluc1} and~\ref{logfluc0}.
Using the notation introduced in 
Definition~\ref{discreteW}, we start by recalling the form of $\P_k^p(\tau_0(W) = \infty)$ and the  Markovian dynamics of $W$ under $\P_k^{p,+}$. (In the sequel we will usually omit the argument $W$ in $\tau_\ell(W)$ and $\sigma_\ell(W)$ when working with the canonical model.)
\begin{remark}\label{condelem}
  \begin{enumerate}[a)]
      \item For $p>\tfrac 12$ and $r:= \tfrac {1-p}p$ one has for $k\in \N$
        \begin{equation} \label{rayknight}
          \mathfrak h_p(k):= \P_k^p(\tau_0 = \infty)= 1-r^k. 
        \end{equation}
      \item Under $\P_\ell^{p,+}$, $W$ is a Markov chain on $\N$, started in $\ell$ and with transition probabilities given by
        \begin{equation}\label{htp}
           \P_\ell^{p,+}(W(n+1) = k+1 \mid W(n) = k)
             = \frac p{\mathfrak h_p(k)} \mathfrak h_p(k+1)
             = \frac{1-r^{k+1}}{(1+r)(1-r^{k})}, \quad k\in \N.  
        \end{equation}
      \item For any $N\in \N$ and $k \in \{1, \ldots, N\}$, the distributions of $(W(0), \ldots, W(\tau_N))$ under $\P_k^p(\cdot \mid \tau_N< \tau_0)$ and under $\P_k^{p,+}$ are equal.
  \end{enumerate}
\end{remark}

\begin{definition}
  \begin{enumerate}[a)]
      \item Complementing the definition in~\eqref{rayknight} for $p=1/2$ by putting
        $$
          \mathfrak h_{1/2}(k)
           := k 
           ,\quad k \in \N,
        $$
        let $W_{1/2}^*$ be the Markov chain on $\N_0$ with transition probabilities $\P(W^*_{1/2}(1) =1 \mid W_{1/2}^*(0) =0) = 1$ \, and
        \begin{equation}\label{defWhalf}
          \P(W_{1/2}^*(1) = k\pm 1 \mid W_{1/2}^*(0) =k)
           = \frac {1/2}{\mathfrak h_{1/2}(k)} \mathfrak h_{1/2}(k\pm 1)
           = \frac 12 \frac{k\pm 1}k, \quad k\in \N.
        \end{equation}
      \item In the sequel (and in accordance with~\eqref{defWhalf}) let $W_p^*$ denote  the Markov chain on $\N_0$ with transition probabilities given by~\eqref{htp}, complemented by $\P(W^*_p(1) =1 \mid W_p^*(0) =0) = 1$.
  \end{enumerate}
\end{definition}

\begin{remark}\label{besremark}
   The Markov chain $W_{1/2}^*$ is a {\em Bessel-like random walk} in the sense of~\cite{alexander}, cf.\ also~\cite[Sec 2]{pitman1974one}. The next lemma guarantees that the upward drift of $W_p^\ast$ is minimal for $p=\tfrac 12$.
\end{remark}

\begin{lemma}\label{driftcompar} For $p> 1/2$ and $k=1,2,\ldots$
\begin{equation}\label{Wpdrift}
 \P(W_p^*(1) = k+1 \mid W_p^*(0) =k)  >  \P(W_{1/2}^*(1) = k+1 \mid W_{1/2}^*(0) =k).
\end{equation}
\end{lemma}
\begin{proof} Because of~\eqref{htp} and\eqref{defWhalf}, the proof of \eqref{Wpdrift} amounts to checking that
for $r< 1$ and $k=1,2,\ldots$
\begin{equation}\label{quizz}
    \frac 12 \frac {k+1}k
      < \frac{1-r^{k+1}}{(1+r)(1-r^{k})}.
\end{equation}
\eqref{quizz}  is equivalent to 
\begin{align*}
    f(r):=
     & -(k-1)r^{k+1}+(k+1)r^k-(k+1)r+k-1 > 0.
\end{align*}
It is readily checked that $f(1)=f'(1)=0$ and $f''(r)>0$ for $r\in(0,1)$, which proves \eqref{quizz}.
\end{proof}
The following is an immediate consequence of Lemma~\ref{driftcompar}.
\begin{cor}\label{downexc}
For $k \in \N$ and $p \in (0,1)$ let $W^*_{1/2}$ and $W^*_p$ both start in $0$. Then, with $\tau_k$ defined in~\eqref{deftauk},
\begin{align*}
  \max_{1\le k \le N} \max_{\tau_{k}(W_p^*)\le n < \tau_{k+1}(W_p^*)}
    & \big(\log_N(k)- \log_N(W_p^*(n)\big)
  \quad \mbox{ is stochastically dominated by }\\
  \max_{1\le k \le N} \max_{\tau_{k}(W_{1/2}^*)\le n < \tau_{k+1}(W_{1/2}^*)}
    & \big(\log_N(k)- \log_N(W_{1/2}^*(n)\big). 
\end{align*}
 \end{cor}
\normalcolor
\begin{prop}\label{HKKcor}
    Let $W^*:= W_{1/2}^*$ start in $0$. Then\\
a) for any $\beta > 0$
\begin{equation}\label{lllaw}
    \sup_{1\le n \le N^\beta} \big|\log_N W^*(n)-\log_N\sqrt n \, \big| \to 0 \quad \mbox{ a.s.\ as } N\to \infty,
\end{equation}
b) with $\tau_k$ as in~\eqref{deftauk},
\begin{equation}\label{fluct}
   \max_{1\le k \le N} \max_{\tau_{k}(W^*)\le n < \tau_{k+1}(W^*)} \big(\log_N(k)- \log_N(W^*(n)\big) \to 0 \quad \mbox{ a.s.\ as } N\to \infty.
\end{equation}
\end{prop}
\begin{proof}
       \cite[Theorem 1]{hambly2003law} shows that almost surely 
     \begin{equation}\label{limsup}
       \limsup_{n\to \infty} \frac {W^*(n)}{\sqrt{2n\log\log n}} = 1, \qquad   \liminf_{n\to \infty} \frac {(\log n)^2W^*(n)}{\sqrt{2n}} = \infty. 
     \end{equation}
     Taking logarithms in~\eqref{limsup}  gives
\begin{equation}\label{littleo}
    \frac{\log W^*(n)  - \log \sqrt{n}}{\log n} \to 0
     \qquad \mbox{a.s. as } n \to \infty.
\end{equation}
\normalcolor
Consequently,  with 
$$
  \varepsilon_N^{}(n)
   := \Big| \frac {\log W^*(n)}{\log N} -  \frac{ \log \sqrt n}{\log N} \Big|
$$
we obtain
\begin{equation}\label{large}
  \sup_{\log N< n \le N^\beta}  \varepsilon_N(n)
    \to 0
     \qquad \mbox{a.s. \, as } N\to \infty.
\end{equation}
On the other hand, using the trivial bound $W^*(n) \le n$,
\begin{equation}\label{small}
  \sup_{1\le n \le \log N} \varepsilon_N^{}(n)
    \le  \Big| \frac {2 \log \log N}{\log N} \Big|
     \quad \mbox{ for all } N \ge 1.
\end{equation}
Combining~\eqref{large} and~\eqref{small} proves part a) of the lemma. To show part b), we choose $\beta>2$ (implying that $\P(\tau_N < N^\beta \mbox{ for sufficiently large } N) =1 $) and abbreviate
\begin{equation}\label{defdelta}\delta_N:= \sup_{1\le n \le N^\beta} \big|\log_N W^*(n)-\log_N\sqrt n \, \big|.
\end{equation}
For the rest of this proof we choose $k \in \{1,\ldots, N\}$ and $j \in \{\tau_k, \tau_k +1, \ldots,\tau_{k+1}-1\}$. It follows from the definition of $\tau_\ell$ in~\eqref{deftauk} that
\begin{equation}\label{firstest}
\log_Nk = \log_N(W^*(\tau_k)) \ge \log_N(W^*(j)), \qquad  \log_N \sqrt {\tau_{k}} \le \log_N \sqrt j,
\end{equation}
On the other hand the definition of $\delta_N$ in~\eqref{defdelta} implies that
\begin{equation}\label{secondest}\big | \log_Nk -\log_N \sqrt {\tau_{k}} \big| \le \delta_N, \qquad \big|\log_N W^*(j)-\log_N\sqrt j \, \big|\le \delta_N.
\end{equation}
\eqref{firstest} and~\eqref{secondest} together readily imply that 
\begin{equation}\label{wanted}\big | \log_Nk -\log_N W^*(j)\big| \le 2 \delta_N.
\end{equation}
From part a) we know that $\delta_N\to 0$ a.s.\ as $N\to \infty$; this completes the proof of assertion b).
   \end{proof}
   As a direct consequence of Corollary~\ref{downexc} and Proposition~\ref{HKKcor} b) we obtain
\begin{cor}\label{corlogfluc}
    The convergence~\eqref{fluctpNnew} holds with $(p(1,N))$ replaced by any sequence $(p(N))$ in $(\tfrac 12, 1)$.
\end{cor}
   This concludes the proof of Lemma~\ref{logfluc1}. 
  We are now going to prepare for the proof of Lemma~\ref{logfluc0} by a time-reversal argument  provided by the next lemma. This lemma is an elementary instance of the general theory on time reversal of Markov processes (\cite{nagasawa1964time}, \cite{chung1969reverse}); for the reader's convenience we include a proof by ``counting paths''. Again we will use the notation from Definition~\ref{discreteW} and Lemma~\ref{logfluc1}.  
\begin{lemma}\label{reversallemma} For $p < 1/2$ and $N \in \N$ let $\P_N^{p,-}$ denote the probability measure $\P_N^p$ conditioned under the event $\{\tau_0< \tau_N^+\}$. Then
$$(W(\tau_0), W(\tau_0-1), \ldots, W(0))\quad \mbox{ \text{under} } \P_N^{p,-}$$
has the same distribution as
$$(W(0), W(1), \ldots, W(\tau_N))\quad \mbox{ \text{under}   } \P_0^{1-p,+}.$$
\end{lemma}
\begin{proof}  Let $\mathbb W_{N,0}$ be the set of paths (with steps $\pm 1$) that start in state $N$ and reach $0$ (in finitely many steps) without returning to state~$N$. Every $w\in \mathbb W_{N,0}$ is thus of the form
    $$w=(w_0, w_1, \ldots, w_{\ell(w)})$$
    with $ w_0 = N$, $w_{\ell(w)}= 0$ and $N-1\ge w_n\ge 1 $ for $0< n < \ell(w)$. We call $\ell(w)$ the {\em length} of $w$; this is  the number of steps taken by $w$ to reach $0$. We write $u(w)$ for the number of upward steps, and $d(w)= \ell(w)-u(w)$ for the number of downward steps in $w$. Defining
    $$\pi(w):= p^{u(w)}q^{d(w)}, \quad w\in \mathbb W_{N,0}.$$
    we observe that
\begin{equation}\label{condpro1}
       \P_N^{p,-}((W(0), \ldots, W(\tau_0)) = w) = \frac{\pi(w)} {\sum\limits_{w'\in \mathbb W_{N,0}} \pi(w')}, \quad w\in \mathbb W_{N,0}. 
    \end{equation}
    The mapping $\rho$ defined by
\begin{equation}\label{defrhoold}
     \rho (w):= \big(w(\tau_0^{}), w(\tau_0^{}-1) \ldots, w(0)\big), \quad w \in \mathbb W_{N,0},  
    \end{equation}
    transforms $\mathbb W_{N,0}$ bijectively into
    $$\mathbb W_{0,N}:= \rho(\mathbb W_{N,0}),$$
    the set of paths (with steps $\pm 1$) that start in state $0$ and reach $N$ (in finitely many steps) without returning to $0$ in between.  Every $v\in \mathbb W_{0,N}$ is thus of the form
    $$
      v
        = (v_0,v_1,\ldots, v_{\ell(v)})
    $$
    with $v_0 =0$, $v_{\ell(v)} = N$ and $0<v_n < N$  for $0< n < \ell(v)$. We define
    $$
      \widehat \pi(v)
       :=  p^{d(v)}q^{u(v)}, \quad v \in \mathbb W_{0,N}
    $$
  and observe that
  \begin{equation}\label{pfdown1}
       \P_0^{1-p, +}((W(0),\ldots, W(\tau_N)) = v) =\frac{\widehat \pi(v)}{\sum_{v'\in \mathbb W_{0,N}} \widehat \pi(v').}
  \end{equation}
    Under the mapping $\rho$, upward (resp.\ downward) steps of $w$ become downward (resp.\ upward) steps of~$\rho(w)$, hence
  \begin{equation}\label{biject1}
      \widehat \pi(v) = \pi(\rho^{-1}(v)), \quad v\in \mathbb W_{0,N}.
  \end{equation}
The claimed equality in distribution now follows readily by combining~\eqref{condpro1}, \eqref{pfdown1} and~\eqref{biject1}. 
\end{proof}
Because time reversal takes first hitting times into times of last visit,  the following corollary is an immediate consequence of Corollary~\ref{corlogfluc} and Lemma~\ref{reversallemma}.

\begin{cor}\label{Besseldown1}
For any sequence $(p(N))$ in $(0, \tfrac 12)$,
\begin{equation}\label{flucdown1}
     \max_{N> k \ge 2}\,\max_{\sigma_{k+1}<n\le \sigma_{k}}(\log_N k - \log_N(W(n))) \to 0 \quad \mbox{ in } \P_N^{p(N),-}\mbox{-probability as } N\to \infty.
\end{equation}
\end{cor}

We now turn to the proof of Lemma~\ref{logfluc0}, and recall that $S(W)$ denotes the maximum of the path~$W$ before it hits $0$.
Since  $k_N = o(N)$ as $N\to \infty$, it is obvious that
\begin{equation}\label{staysmall}
\P_{k_N}^{p(0,N)}(S(W) < N) \ge \P_{k_N}^{1/2}(S(W) < N) \to 1 \quad \mbox{ as } N\to \infty.
\end{equation}
As a consequence of~\eqref{staysmall}, the variation distance between the distribution of $(W(0), \ldots W(\tau_0))$ under~$\P_{k_N}^{p(0,N)}$ and the distribution of $(W(\tau_{k_N}), \ldots, W(\tau_0))$ under $\P_N^{p(0,N),-}$ tends to $0$ as $N\to \infty$. Thus there exist random paths $W_1$ and $W_2$ defined on a common probability space such that 
\begin{itemize}
\item $ (W_1(0), \ldots, W_1(\tau_0))$ has the same distribution as $(W(0), \ldots W(\tau_0))$ under $\P_{k_N}^{p(0,N)}$, 
\item
$ (W_2(\tau_{k_N}), \ldots, W_2(\tau_0))$ has the same distribution as $(W(\tau_{k_N}), \ldots W(\tau_0))$ under $\P_N^{p(0,N),-}$, 
\item
$\P((W_1(0), \ldots, W_1(\tau_0)) \neq (W_2(\tau_{k_N}), \ldots, W_2(\tau_0)))\to 0$ as $N\to \infty$.
\end{itemize}
With
$$
  k^*:=k^*(W_2):= \sup\{\ell\mid \sigma_\ell\geq\tau_{k_N}\},
$$
the path $(W_2(\tau_{k_N}),\ldots, W_2(\sigma_{k^*}))$ is a part of $(W_2(\sigma_{k^*+1}), \ldots, W_2(\sigma_{k^*}))$,~\eqref{flucdownfirst} follows from~\eqref{flucdown1} and the coupling.  In the same manner also~\eqref{flucdown2new} results from~\eqref{flucdown1} combined with the coupling, completing the proof of Lemma~\ref{logfluc0}.

\subsection*{Acknowledgements} 
We thank the Associate Editor, and  two anonymous referees for their excellent reviews. These reviews contained a wealth of constructive remarks, as well as summaries of the paper's content which we found so helpful that we permitted ourselves to take parts of them verbatim into a discussion section.  

Funding acknowledgements by AT:  This paper was supported by the János Bolyai Research Scholarship of the Hungarian Academy of Sciences.  Project no.\ STARTING 149835 has been implemented with the support provided by the Ministry of Culture and Innovation of Hungary from the National Research, Development and Innovation Fund, financed under the STARTING\_24 funding scheme. 

\bibliographystyle{plainnat}
\bibliography{literature_revision}

\end{document}